\documentclass[10pt,reqno]{amsart}

\usepackage{amsmath,amsfonts,amsbsy,amsgen,amscd,mathrsfs,amssymb,amsthm}
\usepackage{enumerate}

\usepackage[usenames,dvipsnames]{xcolor}
\usepackage[colorlinks=true,citecolor=blue,linkcolor=blue]{hyperref}

\newtheorem{thm}{Theorem}[section]
\newtheorem{lem}[thm]{Lemma}

\newtheorem{prop}[thm]{Proposition}
\newtheorem{cor}[thm]{Corollary}

\theoremstyle{definition}

\newtheorem{rem}[thm]{Remark}

\numberwithin{equation}{section} 
\numberwithin{figure}{section}
\numberwithin{table}{section}

\newcommand{\Vol}{\mathrm{Vol}}
\newcommand{\cl}{\mathop{\mathrm{cl}}}
\newcommand{\eqa}{\mathop{\stackrel{\mathrm{a.e.}}{=}}}
\newcommand{\lea}{\mathop{\stackrel{\mathrm{a.e.}}{\le}}}
\newcommand{\gea}{\mathop{\stackrel{\mathrm{a.e.}}{\ge}}}
\newcommand{\supp}{\mathop{\mathrm{supp}}}

\begin{document}

\title{The equality cases of the Ehrhard-Borell inequality}

\author{Yair Shenfeld}
\address{Sherrerd Hall 323, Princeton University, Princeton, NJ
08544, USA}
\email{yairs@princeton.edu}

\author{Ramon van Handel}
\address{Fine Hall 207, Princeton University, Princeton, NJ 
08544, USA}
\email{rvan@princeton.edu}

\begin{abstract}
The Ehrhard-Borell inequality is a far-reaching refinement of
the classical Brunn-Minkowski inequality that captures the sharp convexity 
and isoperimetric properties of Gaussian measures. Unlike in the 
classical Brunn-Minkowski theory, the equality cases in this
inequality are far from evident from the known proofs. The equality 
cases are settled systematically in this paper. An essential ingredient
of the proofs are the geometric and probabilistic properties of certain 
degenerate parabolic equations. The method developed here serves as a 
model for the investigation of equality cases in a broader class of 
geometric inequalities that are obtained by means of a maximum principle.
\end{abstract}

\subjclass[2000]{60G15, 39B62, 52A40, 35K65, 35B50}

\keywords{Ehrhard-Borell inequality; equality cases; Gaussian measures; 
Gaussian Brunn-Minkowski inequalities; degenerate parabolic equations}

\maketitle

\thispagestyle{empty}

\section{Introduction}
\label{sec:intro}

The Brunn-Minkowski inequality plays a central role in numerous 
problems in different areas of mathematics \cite{Gar02,Bar06,Sch93}.
In its simplest form, it states that
$$
	\Vol(\lambda A+\mu B)^{1/n} \ge 
	\lambda\Vol(A)^{1/n} + \mu\Vol(B)^{1/n}
$$
for nonempty closed sets $A,B\subseteq\mathbb{R}^n$ 
and $\lambda,\mu>0$. In the nontrivial case $0<\Vol(A),\Vol(B)<\infty$,
equality holds precisely when $A,B$ are homothetic and convex.

From the very beginning, the study of the cases of equality in the 
Brunn-Minkowski inequality has formed an integral part of the theory. 
Minkowski, who proved the inequality for convex sets, established the 
equality cases in this setting by a careful analysis of the proof 
\cite[p.\ 247, (6)]{Min10}. Both the inequality and equality cases were 
later extended to measurable sets \cite{Lus35,HM53}; the equality 
cases are first shown to reduce to the convex case, for which 
Minkowski's result can be invoked. The understanding of equality cases 
plays an important role in its own right. It provides valuable insight 
into the Brunn-Minkowski theory and into closely related inequalities, 
such as the Riesz-Sobolev inequality \cite{Bur94}. It also guarantees 
uniqueness of solutions to variational problems that arise in 
geometry and mathematical physics (including the classical isoperimetric 
problem), e.g., \cite[\S 6]{Gar02} or \cite[\S 4.1]{Col05}.

This paper is concerned with analogues of the Brunn-Minkowski inequality 
for the standard Gaussian measure $\gamma_n$ on $\mathbb{R}^n$, which 
play an important role in probability theory. The Brunn-Minkowski 
inequality is rather special to the Lebesgue measure: while a weak form 
of the inequality holds for any log-concave measure \cite{Bor74}, this 
property is not sufficiently strong to explain, for example, the 
isoperimetric properties of such measures. It is therefore remarkable
that a sharp analogue of the Brunn-Minkowski inequality proves to exist
for Gaussian measures. A simple form of this inequality is as follows;
in the sequel, we denote $\Phi(x):=\gamma_1((-\infty,x])$.

\begin{thm}[Ehrhard, Borell]
\label{thm:simpleeb}
For closed sets $A,B\subseteq\mathbb{R}^n$ and any $\lambda,\mu>0$
such that $\lambda+\mu\ge 1$ and $|\lambda-\mu|\le 1$, we have
$$
	\Phi^{-1}(\gamma_n(\lambda A+\mu B)) \ge
	\lambda \Phi^{-1}(\gamma_n(A)) +
	\mu \Phi^{-1}(\gamma_n(B)).
$$
If $A,B$ are also convex, the conclusion
remains valid assuming only $\lambda+\mu\ge 1$.
\end{thm}

The Ehrhard-Borell inequality lies at the top of a large hierarchy of 
Gaussian inequalities. It implies the Gaussian isoperimetric inequality, 
which states that half-spaces minimize Gaussian surface area among 
all sets of the same measure; the Gaussian isoperimetric inequality 
in turn implies numerous geometric and analytic inequalities for 
Gaussian measures \cite{Led96,Lat02}. It has recently been understood that 
Theorem \ref{thm:simpleeb} gives rise to new concentration phenomena for 
convex functions that go beyond the isoperimetric theory, see 
\cite{PV17,Val17} and unpublished results of the second author recorded 
in \cite{GMRS17}. Further applications can be found in 
\cite{Ehr83,Bor96,BS97,Ban98,LO99,Hor05,KM17}.

Ehrhard originally discovered the special case of Theorem 
\ref{thm:simpleeb} where $\lambda=1-\mu$ and where $A$ and $B$ are convex 
\cite{Ehr83}. Ehrhard's proof, using a Gaussian analogue of Steiner 
symmetrization, relies heavily on convexity (but see \cite{Lat96}), and 
the question whether the inequality holds for arbitrary sets---as might be 
expected by analogy with the classical Brunn-Minkowski theory---remained 
open for a long time. This question was finally settled by Borell 
\cite{Bor03,Bor08}, who invented an entirely new method of proof to 
establish Theorem \ref{thm:simpleeb} for general sets $A,B$ and to 
identify all values of $\lambda,\mu>0$ for which the inequality is valid. 
It is surprising that the latter is nontrivial and depends on whether or 
not the sets $A,B$ are convex, a feature not present in the classical 
Brunn-Minkowski inequality for which the admissible range of $\lambda,\mu$ 
is trivial by homogeneity of the Lebesgue measure. The fact that 
Gaussian measure lacks the homogeneity and translation-invariance 
properties of Lebesgue measure apparently makes the Gaussian theory much 
more subtle than its classical counterpart. Consequently, even a question 
as basic as the equality cases of Theorem \ref{thm:simpleeb} is far from 
evident from the known proofs and has remained open. This question will be 
systematically settled in this paper. For example, we will prove the 
following.

\begin{thm}
\label{thm:simpleeq}
Let $A,B\subseteq\mathbb{R}^n$ be closed sets with 
$0<\gamma_n(A),\gamma_n(B)<1$
and let $\lambda\ge\mu>0$ satisfy 
$\lambda+\mu\ge 1$ and $\lambda-\mu\le 1$. Then equality in
Theorem \ref{thm:simpleeb}
$$
        \Phi^{-1}(\gamma_n(\lambda A+\mu B)) =
        \lambda \Phi^{-1}(\gamma_n(A)) +
        \mu \Phi^{-1}(\gamma_n(B))
$$
holds if and only if the following hold.
\begin{enumerate}[$\bullet$]
\itemsep\abovedisplayskip
\item If $\lambda\ne 1-\mu$ and $\lambda\ne 1+\mu$, then $A$ and $B$ must
be parallel half-spaces:
\begin{equation}
	\label{eq:H}
	\tag{H}
	A=\{x\in\mathbb{R}^n:\langle a,x\rangle + b \ge 0\},\qquad
	B=\{x\in\mathbb{R}^n:\langle a,x\rangle + c \ge 0\}
\end{equation}
for some $a\in\mathbb{R}^n$ and $b,c\in\mathbb{R}$.
\item If $\lambda = 1-\mu$, then \textbf{either} \eqref{eq:H} holds,
\textbf{or} $A,B$ are convex and $A=B$.
\item If $\lambda = 1+\mu$, then \textbf{either} 
\eqref{eq:H} holds, \textbf{or} $B$ is convex and $-A=\cl B^c$.
\end{enumerate}
\vskip.1cm
If $A,B$ are convex, the conclusion remains valid if the
assumption $\lambda-\mu\le 1$ is omitted; in particular, 
for any $\lambda+\mu>1$, equality holds if and only if \eqref{eq:H} holds.
\end{thm}

It is not hard to verify that each case described in Theorem 
\ref{thm:simpleeq} yields equality in Theorem \ref{thm:simpleeb}. What 
is far from obvious is that these turn out to be the \emph{only} 
equality cases. Theorem \ref{thm:simpleeq} is in fact only a special 
case of our main result, which characterizes all equality cases of the 
functional (Pr\'ekopa-Leindler) form of the Ehrhard-Borell inequality 
for arbitrarily many sets or functions, and which deals with the more 
general measurable situation. The complete characterization is somewhat
more involved and we postpone a full statement to section \ref{sec:main}.

In the last paper written before his death \cite{Ehr86}, Ehrhard 
established the special case of Theorem \ref{thm:simpleeq} where $A,B$ 
are convex and $\lambda=1-\mu$. However, as was the case for the proof 
of the inequality, the general setting of interest in this paper appears 
to be out of reach of Ehrhard's methods. Our starting point will instead 
arise from the ideas introduced by Borell \cite{Bor03}, which we 
presently outline very briefly and which will be recalled in more detail 
in section \ref{sec:borell}. (Several new proofs of the Ehrhard-Borell 
inequality were recently discovered \cite{vH17,NP16,Iv16}, but none of 
these appear to be appropriate for the investigation of equality cases.)

Let us first recall the functional form of Theorem \ref{thm:simpleeb}:
if functions $f,g,h$ satisfy
$$
	\Phi^{-1}(h(\lambda x+\mu y)) \ge
	\lambda\Phi^{-1}(f(x)) + \mu\Phi^{-1}(g(y)),
$$
then 
$$
	\Phi^{-1}\bigg(\int h\, d\gamma_n\bigg) \ge
	\lambda\Phi^{-1}\bigg(\int f\, d\gamma_n\bigg) +
	\mu\Phi^{-1}\bigg(\int g\, d\gamma_n\bigg).
$$
The statement of Theorem \ref{thm:simpleeb} is recovered by setting
$f=1_A$, $g=1_B$, $h=1_{\lambda A+\mu B}$. To prove this functional
inequality, Borell introduces the function
$$
	C(t,x,y) :=
	\Phi^{-1}(Q_th(\lambda x+\mu y)) -
	\lambda\Phi^{-1}(Q_tf(x)) - \mu\Phi^{-1}(Q_tg(y)),
$$
where $Q_t$ denotes the heat semigroup
$$
	Q_tf(x) := \int f(x+\sqrt{t}z)\,\gamma_n(dz).
$$
The assumption on $f,g,h$ can now be written as $C(0,x,y)\ge 0$, while 
the Ehrhard-Borell inequality to be proved can be written as 
$C(1,0,0)\ge 0$. The remarkable observation of Borell is that the 
function $C(t,x,y)$ is the solution of a certain parabolic equation in 
the domain $[0,\infty)\times\mathbb{R}^n\times\mathbb{R}^n$. One can 
therefore invoke the weak parabolic maximum principle \cite{Ev97}, which
states that the minimum of the function $C$ must be attained at the
boundary of its domain. Therefore, $\min C = \min C(0,\cdot,\cdot)\ge 0$
by assumption, and the desired inequality follows.

Borell's approach to the inequality immediately suggests a promising 
approach to the equality cases: one could attempt to invoke the strong 
parabolic maximum principle, which states that the minimum of the 
function $C$ can be attained \emph{only} at the boundary unless the 
function $C$ is constant \cite{Ev97}. We could then reason that if we
have equality $C(1,0,0)=0$ in the Ehrhard-Borell inequality, then $C$
attains its minimum at a non-boundary point and therefore
$C(0,\cdot,\cdot)\equiv 0$, that is,
$$
	\Phi^{-1}(h(\lambda x+\mu y)) =
	\lambda\Phi^{-1}(f(x)) + \mu\Phi^{-1}(g(y)).
$$
For smooth $f,g,h$, this will imply that $\Phi^{-1}(f)$, $\Phi^{-1}(g)$, 
$\Phi^{-1}(h)$ are linear functions
$$
	f(x) = \Phi(\langle a,x\rangle +b),\quad
	g(x) = \Phi(\langle a,x\rangle +c),\quad
	h(x) = \Phi(\langle a,x\rangle +\lambda b+\mu c).
$$
This equality case corresponds exactly to case \eqref{eq:H} in
Theorem \ref{thm:simpleeq}, which is obtained as the limiting case of
these $f,g,h$ by letting $a,b,c\to\infty$ proportionally.

Unfortunately, the reasoning just given must clearly be flawed: 
it suggests that \eqref{eq:H} is the only equality case of the 
Ehrhard-Borell inequality, while we know from Theorem \ref{thm:simpleeq} 
that this is not the case. The error in the above reasoning lies in the 
validity of the strong maximum principle. While the weak maximum 
principle holds for any parabolic equation under minimal regularity 
assumptions, the strong maximum principle holds only when the parabolic 
equation is nondegenerate. An inspection of the equation satisfied by 
Borell's function $C$ will show that the equation is nondegenerate only 
when $\lambda\ne 1-\mu$ and $\lambda\ne 1+\mu$. In the degenerate cases 
$\lambda=1-\mu$ and $\lambda=1+\mu$, the strong maximum principle is no
longer valid, and indeed in each of these cases we see additional equality
cases appearing.

The core of this paper is devoted to showing how each of the equality 
cases of the Ehrhard-Borell inequality can be obtained from the 
analysis of the underlying degenerate parabolic equations, in which
we make extensive use of probabilistic methods.
Our method is 
based on the following fact, which follows by combining the 
Stroock-Varadhan support theorem \cite{SV72} with certain geometric 
arguments \cite{NvdS16}:
a probabilistic form of
the strong maximum principle holds locally near 
a point in the domain if a certain Lie algebra generated by the 
underlying vector fields spans the tangent space (this is closely 
related to H\"ormander's hypoellipticity condition). If we can apply the 
strong maximum principle in any local neighborhood in the domain, then 
analyticity of the heat semigroup will allow us to conclude that the 
solution is constant everywhere and thus the equality case \eqref{eq:H} 
will follow as already indicated above. Using this idea, we arrive at 
the following dichotomy: either the Lie algebra rank condition is 
satisfied at some point in the domain, in which case the equality case 
\eqref{eq:H} follows; or the Lie algebra rank condition fails 
everywhere, which imposes strong constraints on the underlying vector 
fields. A careful analysis of the latter will give rise to the remaining 
equality cases.

There are a number of technical issues that arise when implementing the 
above program that will be addressed in the proof. The analysis of the 
underlying Lie algebra is significantly simplified by considering the 
Ehrhard-Borell inequality on $\mathbb{R}$. We will therefore implement 
most of the ideas outlined above in the one-dimensional case, and then 
extend the conclusion to arbitrary dimension by an induction argument. 
Another issue that will arise is that the function $C(t,x,y)$ may be 
singular at $t=0$ (for example, when $f,g,h$ are indicators), so that we 
cannot reason directly about its behavior on the boundary. We will 
therefore perform our analysis only at times $t>0$ where the heat 
semigroup smooths the functions $f,g,h$, and then use an inversion 
argument to deduce the general equality cases. It is interesting to note 
that this approach is completely different than the analysis of equality 
cases in the classical Pr\'ekopa-Leindler inequality \cite{Dub77}, where 
the functional equality cases are reduced to the set equality cases in 
the Brunn-Minkowski inequality.

As in the classical Brunn-Minkowski theory, the study of equality cases in 
the Ehrhard-Borell inequality provides us with fundamental insight into 
the structure of Gaussian inequalities. At the same time, the maximum 
principle method of Borell has turned out to be a powerful device that 
makes it possible to prove many other geometric inequalities 
\cite{BH09,IV15}. The issues encountered in proving the equality cases of 
the Ehrhard-Borell inequality are characteristic of this method, and the 
techniques developed in this paper serve as a model for the investigation 
of equality cases in situations that may be inaccessible by other methods. 
For example, a natural avenue for further investigation (which we will not 
pursue here) would be to exploit the methods of this paper to investigate 
the equality cases of various nontrivial extensions of Ehrhard's 
inequality developed in \cite{BH09,IV15,Iv16,vH17}.

Once the equality cases have been settled, another natural open problem is 
to obtain quantitative stability estimates (see, e.g., \cite{Fig14}). In 
particular, as is often the case in Gauss space, the Ehrhard-Borell 
inequality is infinite-dimensional in nature \cite{Bor03} and thus 
dimension-free estimates are of particular interest. The methods developed 
in this paper are local in nature, and are therefore unlikely to give 
access to quantitative information. The challenging stability problem for 
the Ehrhard-Borell inequality appears to remain out of reach of any method 
developed to date.

The remainder of this paper is organized as follows. Section 
\ref{sec:main} is devoted to the full statement of our main results, as 
well as some preliminary analysis and notation. In section 
\ref{sec:borell}, we recall Borell's proof of the Ehrhard-Borell 
inequality and outline the basic idea behind the proof of the equality 
cases. Section \ref{sec:nondeg} proves our main result in the 
nondegenerate cases $\lambda\ne 1-\mu$ and $\lambda\ne 1+\mu$. Section 
\ref{sec:basedeg} is devoted to the proof of equality cases in the most 
basic degenerate situation: $\lambda=1-\mu$ in one dimension. This 
section forms the core of the paper. Section \ref{sec:tech} extends the 
basic degenerate case to arbitrary dimension by an induction argument. 
Finally, section \ref{sec:gen} extends the basic cases considered in the 
previous sections to the most general situation covered by our main 
result, completing the proof.

\section{Statement of main results}
\label{sec:main}

Throughout this paper, we will use the following notation. 
We will write 
$\mathbb{\bar R}
:=\mathbb{R}\cup\{-\infty,\infty\}$, where we always use the convention
$\infty-\infty = -\infty+\infty=-\infty$. We denote by 
$\langle\cdot,\cdot\rangle$ and $\|\cdot\|$, respectively, the standard 
Euclidean inner product and norm on $\mathbb{R}^n$. The standard 
Gaussian measure on $\mathbb{R}^n$ will be denoted $\gamma_n$, that is,
$$
	\gamma_n(dx) := \frac{e^{-\|x\|^2/2}}{(2\pi)^{n/2}}\,dx.
$$
We define the standard Gaussian distribution function
$\Phi:\mathbb{\bar R}\to[0,1]$ as
$$
	\Phi(x) := \gamma_1((-\infty,x]) =
	\int_{-\infty}^x \frac{e^{-y^2/2}}{\sqrt{2\pi}}\,dy.
$$
For functions $f,g:\mathbb{R}^n\to\mathbb{\bar R}$, we write $f(x)\eqa g(x)$, 
$f(x)\lea g(x)$, or $f(x)\gea g(x)$ if $f(x)=g(x)$, $f(x)\le g(x)$, 
$f(x)\ge g(x)$, respectively, for almost every $x$ with respect to the 
Lebesgue measure on $\mathbb{R}^n$. A function $f:\mathbb{R}^n\to[0,1]$ 
is said to be \emph{trivial} if $f(x)\eqa 0$ or $f(x)\eqa 1$, and is 
called \emph{nontrivial} otherwise. A function 
$f:\mathbb{R}^n\to\mathbb{\bar R}$ is said to be \emph{a.e.\ concave} if 
$f(x)\eqa\tilde f(x)$ for some concave function $\tilde f:\mathbb{R}^n\to
\mathbb{\bar R}$.

The aim of this section is to give a complete statement of our 
characterization of the equality cases in the Ehrhard-Borell
inequality. Before we do so, we first state the Ehrhard-Borell
inequality in its most general form.

\begin{thm}[Ehrhard, Borell]
\label{thm:be}
Let $\lambda_1,\ldots,\lambda_m>0$ satisfy
\begin{equation}
\label{eq:A}
\tag{A}
	\sum_{i\le m}\lambda_i \ge 1,\qquad\quad
	2\max_{i\le m}\lambda_i \le 1+\sum_{i\le m}\lambda_i.
\end{equation}
Then for any measurable functions
$h,f_1,\ldots,f_m:\mathbb{R}^n\to[0,1]$ such that
\begin{equation}
\label{eq:B}
\tag{B}
	\Phi^{-1}\Bigg(h\Bigg(\sum_{i\le m}\lambda_ix_i\Bigg)\Bigg)
	\gea \sum_{i\le m}\lambda_i\Phi^{-1}(f_i(x_i)),
\end{equation}
we have
$$	\Phi^{-1}\bigg(\int h\,d\gamma_n\bigg)\ge
	\sum_{i\le m}\lambda_i\Phi^{-1}\bigg(
	\int f_i\,d\gamma_n
	\bigg).
$$
If the functions $\Phi^{-1}(h),\Phi^{-1}(f_1),\ldots,\Phi^{-1}(f_m)$
are also a.e.\ concave, the conclusion remains valid when the second
inequality in \eqref{eq:A} is omitted.
\end{thm}

Theorem \ref{thm:be} is given in \cite{Bor08}, where it is also shown 
that the assumption \eqref{eq:A} is necessary. The result stated here 
is actually slightly more general, as we have only assumed that the 
inequality in \eqref{eq:B} holds a.e., while it is assumed in 
\cite{Bor08} that the inequality holds everywhere. The elimination of 
the null set is routine and follows immediately by invoking the results 
of \cite{Dub77} or \cite[Appendix]{BL76}. The significance of this 
apparently minor generalization will be discussed in Remark \ref{rem:null}
below.

We are now ready to state the main result of this paper.

\begin{thm}[Main equality cases]
\label{thm:main}
Let $\lambda_1\ge \lambda_2\ldots,\lambda_m>0$ satisfy \eqref{eq:A},
that is,
$$
	\sum_{1\le i\le m}\lambda_i \ge 1,\qquad\quad
	\lambda_1-
	\sum_{2\le i\le m}\lambda_i \le 1.
$$
Let $h,f_1,\ldots,f_m:\mathbb{R}^n\to[0,1]$ be nontrivial measurable
functions satisfying \eqref{eq:B}. Then
equality in the Ehrhard-Borell inequality
$$
	\Phi^{-1}\bigg(\int h\,d\gamma_n\bigg)=
	\sum_{i\le m}\lambda_i\Phi^{-1}\bigg(
	\int f_i\,d\gamma_n
	\bigg)
$$
holds if and only if the following hold.
\begin{enumerate}[$\bullet$]
\itemsep\abovedisplayskip
\item If $\sum_i\lambda_i\ne 1$ and $\lambda_1-\sum_{i\ge 2}\lambda_i\ne 1$,
then \textbf{either}
\begin{equation}
\label{eq:H1}\tag{H1}
	h(x) \eqa \Phi(\langle a,x\rangle + b),\qquad
	f_i(x) \eqa \Phi(\langle a,x\rangle + b_i)
\end{equation}
for all $i\le m$, \textbf{or}
\begin{equation}
\label{eq:H2}\tag{H2}
	h(x) \eqa 1_{\langle a,x\rangle + b\ge 0},\qquad
	f_i(x) \eqa 1_{\langle a,x\rangle + b_i\ge 0}
\end{equation}
for all $i\le m$, for some $a\in\mathbb{R}^n$, $b_1,\ldots,b_m\in\mathbb{R}$,
and $b=\sum_i\lambda_ib_i$.
\item If $\sum_i\lambda_i=1$, then \textbf{either} \eqref{eq:H1} or
\eqref{eq:H2} hold, \textbf{or} 
$$
	h(x)\eqa f_1(x)\eqa\ldots\eqa f_m(x)
$$
and $\Phi^{-1}(h),\Phi^{-1}(f_1),\ldots,\Phi^{-1}(f_m)$ are a.e.\ concave.
\item If $\lambda_1-\sum_{i\ge 2}\lambda_i=1$, then \textbf{either}
\eqref{eq:H1} or \eqref{eq:H2} hold, \textbf{or}
$$
	1-h(-x)\eqa 1-f_1(-x)\eqa f_2(x)\eqa\ldots\eqa f_m(x)
$$
and $\Phi^{-1}(f_2),\ldots,\Phi^{-1}(f_m)$ are a.e.\ concave.
\end{enumerate}
\vskip.1cm
If the functions $\Phi^{-1}(h),\Phi^{-1}(f_1),\ldots,\Phi^{-1}(f_m)$
are a.e.\ concave, the conclusion remains valid if  the second
inequality in \eqref{eq:A} is omitted. In particular, in this case,
for any $\sum_i\lambda_i>1$, equality holds if and only if either
\eqref{eq:H1} or \eqref{eq:H2} holds.
\end{thm}

In Theorem \ref{thm:main} we assumed that the functions $h,f_1,\ldots,f_m$ 
are all nontrivial. If any of the functions is trivial, then the 
equality cases are trivially verified and no analysis is needed. We list 
these cases separately in the following lemma, which completes our 
characterization of the equality cases of Theorem \ref{thm:be}.

\begin{lem}[Trivial equality cases]
\label{lem:triv}
Let $\lambda_1,\ldots,\lambda_m>0$ and 
$h,f_1,\ldots,f_m:\mathbb{R}^n\to[0,1]$ be measurable
functions satisfying \eqref{eq:B}, at least one of which is trivial.
Then
$$      \Phi^{-1}\bigg(\int h\,d\gamma_n\bigg)=
        \sum_{i\le m}\lambda_i\Phi^{-1}\bigg(
        \int f_i\,d\gamma_n
        \bigg)
$$
if and only if \textbf{either} $h\eqa 1$, at least one of the functions
$f_i$ satisfies $f_i\eqa 1$, and none of the $f_i$ satisfy $f_i\eqa 0$;
\textbf{or} $h\eqa 0$ and at least one of the $f_i$ satisfies $f_i\eqa 0$.
\end{lem}

\begin{proof}
If any of the functions $f_i$ is trivial, then $\Phi^{-1}(\int f_i d\gamma_n)
=\pm\infty$. Thus the equality assumption implies that 
$\Phi^{-1}(\int h d\gamma_n)=\pm\infty$, so $h$ must be trivial as well. 
Thus we need only consider
the cases $h\eqa 1$ and $h\eqa 0$. The conclusion is now readily verified
using the convention $\infty-\infty=-\infty+\infty = -\infty$.
\end{proof}

Let us verify that Theorem \ref{thm:simpleeq} follows from our general 
characterization.

\begin{proof}[Proof of Theorem \ref{thm:simpleeq}]
The `if' part is readily verified using $\gamma_n\{x:\langle a,x\rangle+b\ge 0\}
= \Phi(b/\|a\|)$, that $\lambda A+(1-\lambda)A=A$ if $A$ is convex, and
that $(1+\mu)\cl B^c-\mu B = \cl B^c$ if $B$ is closed and convex.
It therefore remains to prove the `only if' part.

Let $\lambda_1=\lambda$, $\lambda_2=\mu$, $f_1=1_A$, $f_2=1_B$, 
and $h=1_{\lambda A+\mu B}$. It is readily verified that \eqref{eq:A} 
and \eqref{eq:B} are satisfied. Moreover, the assumption 
$0<\gamma_n(A),\gamma_n(B)<1$ and Lemma \ref{lem:triv} rule out 
the trivial equality cases. The equality cases will therefore follow 
from Theorem \ref{thm:main}. Note that \eqref{eq:H1} is ruled out as 
$f_1,f_2,h$ are indicator functions. Moreover, $\Phi^{-1}(1_C)$ is a.e.\ 
concave if and only if $C$ differs by a set of measure zero from a 
convex set. It therefore follows that $A,B$ must satisfy the conclusions
of Theorem \ref{thm:simpleeq} \emph{modulo sets of measure zero}. 

The only special feature of Theorem \ref{thm:simpleeq} that does not follow
directly from our main result is the elimination of the measure zero sets.
This can be done in this case because we assumed that $A,B$ are closed.
For sake of illustration, let us sketch the argument in the case
$\lambda\ne 1-\mu$ and $\lambda\ne 1+\mu$. Define
$$
	\tilde A=\{x\in\mathbb{R}^n:\langle a,x\rangle + b \ge 0\},\qquad
	\tilde B=\{x\in\mathbb{R}^n:\langle a,x\rangle + c \ge 0\}
$$
for $a,b,c$ as in Theorem \ref{thm:simpleeq}. What we have shown so far
is that $A,B$ must differ from $\tilde A,\tilde B$ by sets of measure
zero, that is, $A=(\tilde A\backslash N_{\tilde A})\cup N_{{\tilde A}^c}$ and 
$B=(\tilde B\backslash N_{\tilde B})\cup N_{{\tilde B}^c}$, where we 
used the notation $N_C$ to denote a null set contained in $C$. First 
note that we must have $N_{\tilde A}=N_{\tilde B}=\varnothing$ as
$A,B$ are closed. Now suppose that $N_{\tilde A^c}\ne\varnothing$.
Then there exists $z\in N_{\tilde A^c}$ 
such that $\langle a,z\rangle + b'=0$ for some $b'>b$. But then $\lambda 
A+\mu B\supseteq \{x\in\mathbb{R}^n:\langle a,x\rangle + \lambda b' +\mu 
c\ge 0\}$, which has strictly larger measure than $\lambda\tilde 
A+\mu\tilde B$. As this would violate the equality assumption, we have 
shown that $N_{\tilde A^c}=N_{\tilde B^c}=\varnothing$, completing the proof 
of the first case of Theorem \ref{thm:simpleeq}. The null set of the
remaining cases can be eliminated in the same manner.
\end{proof}

\begin{rem}
\label{rem:null}
The proof of Theorem \ref{thm:simpleeq} illustrates why Theorem 
\ref{thm:be} should be formulated in terms of a.e.\ inequality in 
\eqref{eq:B}. If we were to require that the inequality in \eqref{eq:B} 
holds everywhere, as was done in \cite{Bor08}, the applicability of 
Theorem \ref{thm:be} would become sensitive to null sets: if the 
functions $h,f_i$ are modified on a set of measure zero, they may no 
longer satisfy \eqref{eq:B} everywhere. This sensitivity to null sets is 
inherent to the classical formulation of the Brunn-Minkowski inequality: 
if we modify sets $A,B$ on a null set, their sum $A+B$ could nonetheless 
change on a set of positive measure (or even become nonmeasurable). This 
is problematic for the characterization of the equality cases, as the 
proof of Theorem \ref{thm:main} can only yield candidate equality cases 
up to a.e.\ equivalence, and thus cannot fully characterize the equality 
cases in the strict formulation of the inequality. The latter are 
amenable to a separate analysis of the null sets, as is illustrated in 
the proof of Theorem \ref{thm:simpleeq}. However, as Theorem 
\ref{thm:be} remains valid in its a.e.\ formulation, such null set 
issues should be viewed as an artefact of choosing an unnatural 
formulation of the underlying inequality, rather than being fundamental 
to this theory. We have therefore adopted the viewpoint of 
\cite{BL76,Dub77} that Brunn-Minkowski type inequalities are most 
naturally formulated in a form that is insensitive to null sets. One can 
analogously replace the sum $A+B$ in the Brunn-Minkowski inequality by 
the essential sum $A\oplus B$ which is insensitive to null sets, see 
\cite{BL76} for precise statements.
\end{rem}

The remainder of this paper is devoted entirely to the proof of Theorem 
\ref{thm:main}. The `if' part of the theorem is easy and we dispose of 
it presently. The `only if' part is the core of the theorem, and the 
following sections are devoted to its analysis.

\begin{proof}[Proof of sufficiency in Theorem \ref{thm:main}]
We must verify that each of the cases listed in the theorem statement is an 
equality case of Theorem \ref{thm:be}. That is, we must verify
that \eqref{eq:B} holds, and that the conclusion holds with equality.
\begin{enumerate}[$\bullet$]
\itemsep\abovedisplayskip
\item
\textbf{Case \eqref{eq:H1}.} It is easily seen that \eqref{eq:B} holds.
To verify equality, note that
$$
	\int \Phi(\langle a,x\rangle+b)\,\gamma_n(dx) =
	\int 1_{y\le\langle a,x\rangle+b}\,
	\gamma_1(dy) \, \gamma_n(dx) =
	\Phi(b/\sqrt{1+\|a\|^2}).
$$
\item
\textbf{Case \eqref{eq:H2}.} Let $\tilde h(x)=1_{\langle a,x\rangle+b}$, 
$\tilde f_i=1_{\langle a,x\rangle+b_i}$. Then \eqref{eq:B} holds for 
$\tilde h,\tilde f_i$: if any $\tilde f_i(x_i)=0$, then \eqref{eq:B} 
holds as its right-hand side equals $-\infty$; while if all $\tilde 
f_i(x_i)=1$, then \eqref{eq:B} holds as $\tilde 
h(\sum_i\lambda_ix_i)=1$. As $h\eqa\tilde h$ and $f_i\eqa\tilde f_i$, we 
have verified that \eqref{eq:B} holds for $h,f_i$ as well. To verify 
equality, note that 
$$
	\int 1_{\langle a,x\rangle + b \ge 0}\,\gamma_n(dx) =
	\Phi(b/\|a\|).
$$
\item \textbf{Case $\sum_i\lambda_i=1$.} We already verified that \eqref{eq:H1}
and \eqref{eq:H2} are equality cases. For the remaining case, let
$V$ be a concave function so that $h\eqa f_i\eqa\Phi(V)$.
As $\sum_i\lambda_i=1$, \eqref{eq:B} follows immediately from the concavity
of $V$. On the other hand, equality follows as $\sum_i\lambda_i=1$ and
as $h,f_i$ all coincide a.e.
\item \textbf{Case $\lambda_1-\sum_{i\ge 2}\lambda_i=1$.} 
We already verified that \eqref{eq:H1}
and \eqref{eq:H2} are equality cases. For the remaining case,
let $V$ be a concave function so that $f_2\eqa\Phi(V)$.
Substituting the stated relations between $h,f_i$ 
into \eqref{eq:B} and using the identity
$\Phi^{-1}(1-x)=-\Phi^{-1}(x)$, we
find that \eqref{eq:B} is equivalent in this case to
$$
	-V\Bigg(-\sum_{1\le i\le m}\lambda_ix_i\Bigg)
	\gea -\lambda_1 V(-x_1) +
	\sum_{2\le i\le m}\lambda_i V(x_i).
$$
Rearranging this inequality gives
$$
	V(-x_1) \gea
	\mu_1V\Bigg(-\sum_{1\le i\le m}\lambda_ix_i\Bigg)
	+
	\sum_{2\le i\le m}\mu_i V(x_i),
$$
where $\mu_1=1/\lambda_1$ and $\mu_i=\lambda_i/\lambda_1$ for $i\ge 2$
(one may verify using the convention $\infty-\infty=-\infty+\infty=-\infty$
that this claim is valid even when some of the terms may take the values 
$\pm\infty$). As $\lambda_1-\sum_{i\ge 2}\lambda_i=1$ implies $\sum_i\mu_i=1$, 
\eqref{eq:B} follows from the concavity of $V$. Equality is 
easily verified using 
$$
	\Phi^{-1}\bigg(\int (1-h(-x))\,d\gamma_n\bigg) =
	-\Phi^{-1}\bigg(\int h\,d\gamma_n\bigg),
$$
where we used again $\Phi^{-1}(1-x)=-\Phi^{-1}(x)$ and that 
$\gamma_n$ is symmetric.
\end{enumerate}
\vskip.1cm
This completes our verification of the `if' part of Theorem 
\ref{thm:main}.
\end{proof}

\section{Borell's construction}
\label{sec:borell}

The aim of this section is to recall the construction that appears in 
Borell's proof of Theorem \ref{thm:be}, which forms the starting point 
for our analysis. We will consider only the simplest case of $m=2$ 
functions, which is what will be needed in the sequel. Almost all of our 
analysis will be done in the case of $m=2$ functions, and we return to 
the case $m>2$ in section \ref{sec:gen} at the end of the paper.

Let $\lambda,\mu>0$ and $f,g,h:\mathbb{R}^n\to[0,1]$ be nontrivial
measurable functions.
We define the function $u_f$ in terms of the heat semigroup $Q_t$ by
$$
	u_f(t,x) := \Phi^{-1}(Q_tf(x)),\qquad
	Q_tf(x) := \int f(x+\sqrt{t}z)\,\gamma_n(dz).
$$
The functions $u_g,u_h$ are definite analogously.
A central role throughout this paper will be played by the following
function:
$$
	C(t,x,y) := u_h(t,\lambda x+\mu y) -
	\lambda u_f(t,x) - \mu u_g(t,y).
$$
The key observation of Borell is that $C$ solves a parabolic equation.

\begin{lem}
\label{lem:para}
For every $t>0$, we have
$$
	\frac{\partial C(t,x,y)}{\partial t} =
	\frac{1}{2}\Delta_\rho C(t,x,y) 
	+\langle b(t,x,y),\nabla C(t,x,y)\rangle
	-\frac{1}{2}\|\nabla u_h(t,\lambda x+\mu y)\|^2C(t,x,y),
$$
where $\rho := (1-\lambda^2-\mu^2)/2\lambda\mu$,
$$
	\Delta_\rho := 
	\sum_{i=1}^n
	\Bigg\{
	\frac{\partial^2}{\partial x_i^2} +
	2\rho
	\frac{\partial^2}{\partial x_i \partial y_i} +
	\frac{\partial^2}{\partial y_i^2}
	\Bigg\},
$$
and
$$
	b(t,x,y) :=
	-\frac{1}{2}\left[
	\begin{array}{c}
	u_f(t,x)(\nabla u_h(t,\lambda x+\mu y)+\nabla u_f(t,x)) \\
	u_g(t,y)(\nabla u_h(t,\lambda x+\mu y)+\nabla u_g(t,y))
	\end{array}
	\right].
$$
\end{lem}

\begin{proof}
For any nontrivial function $f:\mathbb{R}^n\to[0,1]$, the function 
$(t,x)\mapsto Q_tf(x)$ takes values in $(0,1)$, is smooth and satisfies 
the heat equation $\frac{d}{dt}Q_tf=\frac{1}{2}\Delta Q_tf$
on $(0,\infty)\times\mathbb{R}^n$. The conclusion follows by a
straightforward, if somewhat tedious, application of the chain rule
of calculus, see \cite{Bor03} for the detailed computation.
\end{proof}

Observe that the equation for $C$ given in Lemma \ref{lem:para} is 
parabolic if and only if $-1\le\rho\le 1$, which is equivalent to the 
conditions $\lambda+\mu\ge 1$ and $|\lambda-\mu|\le 1$ of Theorem 
\ref{thm:simpleeb}. This is the only property of this equation that is 
needed to prove the Ehrhard-Borell inequality; the precise form of the 
different terms in the equation is irrelevant for this purpose. However, 
the specific form of the vector field $b$ will be essential later on for 
our characterization of the equality cases. \emph{From now on until 
section \ref{sec:gen}, we will always assume that $\lambda+\mu\ge 1$ and 
$|\lambda-\mu|\le 1$.}

\begin{rem}
While the proof of Lemma \ref{lem:para} is a matter of explicit 
computation, it may appear somewhat miraculous that the function $C$ 
happens to satisfy a parabolic equation under the assumptions of the 
Ehrhard-Borell inequality. Such phenomena however arise rather broadly
in the study of geometric partial differential equations, see \cite{And15} 
for a survey. In the present case, a probabilistic explanation of the 
hidden convexity underlying this computation and of the significance of 
the conditions $\lambda+\mu\ge 1$ and $|\lambda-\mu|\le 1$ is developed in 
\cite{vH17}. While these different viewpoints may be illuminating for
understanding the origin of this phenomenon, the concrete
statement of Lemma \ref{lem:para} is all that will be needed for the 
purposes of this paper.
\end{rem}

A technical difficulty that is faced both in the proof of the 
Ehrhard-Borell inequality and in our proof of the equality cases is that 
the function $C$ may be singular at $t=0$. For example, when $f=1_A$, 
$g=1_B$, and $h=1_{\lambda A+\mu B}$ as in the proof of Theorem 
\ref{thm:simpleeb}, the function $C$ blows up as $t\downarrow 0$. Fortunately,
$C$ is highly regular for $t>0$. Let us begin by collecting some useful 
regularity properties.

\begin{lem}
\label{lem:anal}
For any $t>0$, the function $x\mapsto u_f(t,x)$ (and $u_g,u_h$) is
real analytic and Lipschitz with constant $t^{-1/2}$.
In particular, $(x,y)\to C(t,x,y)$ is real analytic.
\end{lem}

\begin{proof}
The Lipschitz property of $u_f(t,\cdot)$ is given in
\cite[p.\ 423]{BGL14}. Next, note that $Q_tf$ is real analytic as
we have the polynomial expansion
\begin{align*}
	Q_tf(x) &= 
	\int f(\sqrt{t}y) e^{\langle x,y\rangle/\sqrt{t}-\|x\|^2/2t} 
	d\gamma_n(y) \\ &=
	e^{-\|x\|^2/2t}
	\sum_{n=0}^\infty\frac{t^{-n/2}}{n!}
	\int f(\sqrt{t}y) \langle x,y\rangle^n
	d\gamma_n(y) 
\end{align*}
which is absolutely convergent as
$$
	\int |f(\sqrt{t}y) \langle x,y\rangle^n|
	d\gamma_n(y) 
	\le
	\int |\langle x,y\rangle|^nd\gamma_n(y)
	\le \|x\|^n (C\sqrt{n})^n
$$
for a universal constant $C$ (while $n!\gtrsim (C'n)^n$ by Stirling).
Analogously, as
$\Phi(x)=Q_11_{[0,\infty)}(x)$, the function $\Phi$ is real analytic
as well. As $\Phi$ is strictly increasing, it follows that $\Phi^{-1}$
is real analytic on $(0,1)$ by the inverse function theorem. Thus
$u_f(t,\cdot) := \Phi^{-1}(Q_tf)$ is real analytic.
Analyticity of $C$ follows trivially.
\end{proof}

The regularity provided by Lemma \ref{lem:anal} is not sufficient to 
prove the Ehrhard-Borell inequality; to apply the weak parabolic maximum 
principle, one must impose additional regularity by assuming initially 
that the functions $f,g,h$ are sufficiently smooth, and the general case 
is subsequently obtained by an approximation argument. It will however be 
difficult to obtain equality cases in this manner, as we cannot ensure 
that the equality cases are preserved by approximation. To establish the 
equality cases, we will therefore avoid approximating $f,g,h$ but rather 
work directly with the original functions. For this reason the minimal 
regularity provided by Lemma \ref{lem:anal} will play an important role 
in the sequel.

As was already noted in \cite{Bor03}, the maximum principle admits a 
probabilistic interpretation through the Feynman-Kac formula. As the 
probabilistic approach will play an important role in our analysis of 
the degenerate equality cases, we adopt this viewpoint throughout. To 
this end, recall the following representation of the function $C$. (The 
requisite probabilistic background can be found in \cite{KS88}.)

\begin{lem}
\label{lem:fk}
Consider the stochastic differential equation
$$
	d\left[
	\begin{matrix}
	X_t \\ Y_t
	\end{matrix}
	\right] =
	b(1-t,X_t,Y_t)\,dt + d\left[
	\begin{matrix}
	W_t\\
	B_t
	\end{matrix}
	\right],\qquad
	X_0=Y_0=0
$$
for $\mathbb{R}^n$-valued processes $X_t,Y_t$,
where $W_t,B_t$ are $n$-dimensional standard Brownian motions
with correlation $\langle W^i,B^j\rangle_t = \delta_{ij}\rho t$ and
$b,\rho$ are as defined in Lemma \ref{lem:para}.
This equation has a unique solution for $t\in[0,1)$, and we have for
any $t\in[0,1)$
\begin{align*}
	&
	\Phi^{-1}\bigg(\int h\, d\gamma_n\bigg) -
	\lambda\Phi^{-1}\bigg(\int f\, d\gamma_n\bigg) -
	\mu\Phi^{-1}\bigg(\int g\, d\gamma_n\bigg) 
	\\ &= C(1,0,0) =
	\mathbf{E}\bigg[
	C(1-t,X_t,Y_t)\,
	e^{-\frac{1}{2}\int_0^t\|\nabla u_h(1-s,\lambda
	X_s+\mu Y_s)\|^2ds}
	\bigg].
\end{align*}
If in addition $f,g,h$ take values in $[\delta,1-\delta]$ for some
$0<\delta<1$ and are smooth with bounded first and second derivatives,
then the equation admits a unique solution and the above representation
holds for all $t\in[0,1]$.
\end{lem}

\begin{proof}
Let $t\in[0,1)$. By the elementary properties of the heat semigroup, the 
function $(s,x,y)\mapsto b(1-s,x,y)$ is smooth on $[0,t]\times 
\mathbb{R}^n \times \mathbb{R}^n$, and is therefore locally Lipschitz.
Moreover, note that $b$ satisfies the linear growth property
\begin{align*}
	\sup_{s\le t}\|b(1-s,x,y)\| &\le (1-t)^{-1/2}
	\sup_{s\le t}\{|u_f(1-s,x)|+|u_g(1-s,y)|\}
	\\ &\le
	(1-t)^{-1}\bigg\{\sup_{s\le t}(|u_f(1-s,0)|+
	|u_g(1-s,0)|)+\|x\|+\|y\|\bigg\},
\end{align*}
where we have used twice the Lipschitz property of Lemma \ref{lem:anal}.
Thus existence and uniqueness of the solution $(X_t,Y_t)$ for $t\in[0,1)$
follows from a standard result \cite[Theorem V.12.1]{RW94}. The stochastic
representation of $C(1,0,0)$ now follows from the Feynman-Kac theorem
\cite[Theorem 5.7.6]{KS88}. Under the additional regularity assumptions
stated at the end of the lemma, it is readily verified that $b$ is
globally Lipschitz uniformly on $t\in[0,1]$, and thus both existence
and uniqueness and the Feynman-Kac representation extend to the entire
interval $t\in[0,1]$.
\end{proof}

The parabolic equation for the function $C$ and its probabilistic 
representation form the starting point for the methods developed in this 
paper. To conclude this section, let us briefly sketch how this 
construction gives rise to the Ehrhard-Borell inequality, and outline 
the basic ideas behind the proof of the equality cases.

With the above representation in hand, the Ehrhard-Borell inequality
becomes nearly trivial modulo a technical approximation argument. Suppose that
the functions $f,g,h$ satisfy the stronger regularity assumptions
stated in Lemma \ref{lem:fk}, and that the assumption of the Ehrhard-Borell
inequality
$$
	C(0,x,y) =
	\Phi^{-1}(h(\lambda x+\mu y)) - 
	\lambda\Phi^{-1}(f(x)) -
	\mu\Phi^{-1}(g(y)) \ge 0
$$
holds. Then the conclusion 
\begin{align*}
	&\Phi^{-1}\bigg(\int h\, d\gamma_n\bigg) -
	\lambda\Phi^{-1}\bigg(\int f\, d\gamma_n\bigg) -
	\mu\Phi^{-1}\bigg(\int g\, d\gamma_n\bigg) 
	\\ &
	=\mathbf{E}\bigg[
	C(0,X_1,Y_1)\,
	e^{-\frac{1}{2}\int_0^1\|\nabla u_h(1-s,\lambda
	X_s+\mu Y_s)\|^2ds}
	\bigg] \ge 0
\end{align*}
follows trivially from the probabilistic representation of
Lemma \ref{lem:fk} for $t=1$. The extension of the conclusion to
general $f,g,h$ is then accomplished by an approximation argument
that is unrelated to the methods of this paper, and which we therefore
omit; see \cite[Remark 3.4]{vH17} and the references therein.

This simple argument not only proves the Ehrhard-Borell inequality, but 
also suggests a potential approach to its equality cases.
In order to highlight the main ideas, let us ignore regularity issues
for now (they will be addressed in the proof).
Assume that $f,g,h$ are sufficiently regular 
that the representation of Lemma \ref{lem:fk} holds for all $t\in[0,1]$, 
and that the assumption $C(0,x,y)\ge 0$ of the Ehrhard-Borell inequality 
holds. Suppose now that equality holds in the Ehrhard-Borell inequality
$$
	\Phi^{-1}\bigg(\int h\, d\gamma_n\bigg) -
	\lambda\Phi^{-1}\bigg(\int f\, d\gamma_n\bigg) -
	\mu\Phi^{-1}\bigg(\int g\, d\gamma_n\bigg) = 0.
$$
Then the probabilistic representation of Lemma \ref{lem:fk} immediately
implies that
$$
	C(0,X_1,Y_1) = 0 \quad\mbox{a.s.}
$$
It would appear at first sight that we are almost done, as it is not 
difficult to show (as is done in section \ref{sec:basedeg} below) 
that $C(0,x,y)=0$ implies that \eqref{eq:H1} in Theorem \ref{thm:main} 
holds. However, this conclusion is not correct, as \eqref{eq:H1} is not 
the only equality case of the Ehrhard-Borell inequality even when 
restricted to regular $f,g,h$.

The problem with the above reasoning is that it is not true that 
$C(0,X_1,Y_1)=0$ a.s.\ implies $C(0,x,y)=0$ for all $x,y$, as the law of 
the random vector $(X_1,Y_1)$ need not be supported on all of 
$\mathbb{R}^n\times\mathbb{R}^n$. Indeed, we will see that the other 
equality cases in the Ehrhard-Borell inequality can arise precisely 
when the support of $(X_1,Y_1)$ becomes degenerate. To illustrate this 
phenomenon, it is instructive to consider, for example, the equality 
case where $\lambda=1-\mu$ and $f=g=h$. In this case, $\rho=1$ so that 
the Brownian motions in Lemma \ref{lem:fk} are perfectly correlated 
$W_t=B_t$. Using this fact and the explicit expression for the vector 
field $b$ given in Lemma \ref{lem:para}, it is readily verified that if 
$X_t=Y_t$ and $X_t$ is the solution of the equation
$$
	dX_t = -u_f(1-t,X_t)\nabla u_f(1-t,X_t)\,dt
	+ dW_t,\qquad X_0=0,
$$
then $(X_t,Y_t)$ solves the equation in Lemma 
\ref{lem:fk}. In particular, in this case, $(X_1,Y_1)$ is supported on 
the diagonal $\{(x,x):x\in\mathbb{R}^n\}\subset 
\mathbb{R}^n\times\mathbb{R}^n$, so we could conclude at best that 
$C(0,x,x)=0$ and \emph{not} that $C(0,x,y)=0$. In this way, the 
degeneracy of the support makes it possible for equality cases other 
than \eqref{eq:H1} to appear.

Motivated by these observations, our analysis of the equality cases in 
the Ehrhard-Borell inequality will involve an analysis of the support of 
the solution $(X_t,Y_t)$ of the equation in Lemma \ref{lem:fk}. The 
support of the solution of a stochastic differential equation is 
characterized in principle by the classical support theorem of Stroock 
and Varadhan \cite{SV72} in terms of the solutions of certain ordinary 
differential equations. However, the differential equation that appears 
in Lemma \ref{lem:fk} is not easy to analyze directly, as its drift 
vector field $b$ is itself defined in terms of the functions 
$u_f,u_g,u_h$ that we are attempting to analyze in the first place. 
Moreover, even if we could characterize the support of the equation, it 
is not immediately clear how that would give rise to specific equality 
cases other than \eqref{eq:H1}.

Instead, our proof will indirectly utilize the dichotomy between 
nondegenerate and degenerate support to characterize the different 
equality cases. Roughly speaking, our argument will work as follows. If 
$(X_t,Y_t)$ has nondegenerate support, we will be able to argue as above 
to conclude the equality case \eqref{eq:H1}. Conversely, if $(X_t,Y_t)$ 
has degenerate support, this imposes strong constraints on the 
underlying differential equation: its vector field must always lie in 
the tangent space of a lower-dimensional submanifold. Using methods 
borrowed from nonlinear control theory \cite{NvdS16}, we can translate 
this geometric constraint into an algebraic identity for the vector 
field $b$ which provides the key to the analysis of the remaining 
equality cases. This basic outline of the proof, together with the 
modifications needed to circumvent the regularity issues, will be 
developed in detail in the following sections.

\begin{rem}
In the introduction (section \ref{sec:intro}), we explained the 
Ehrhard-Borell inequality as arising from the weak parabolic maximum 
principle, while we have exploited in this section an alternative 
probabilistic argument. These two approaches are not really distinct, 
however: in essence, the argument given above is little more than a 
probabilistic proof of the weak parabolic maximum principle by means of 
the Feynman-Kac formula (it is a simple exercise to verify that the same 
argument recovers the general form of the maximum principle given, for 
example, in \cite{Ev97}). Similarly, we can view the statement that 
$C(1,0,0)=0$ implies $C(t,x,y)=0$ for all $(x,y)$ in the support of the 
law of $(X_{1-t},Y_{1-t})$ as a probabilistic form of the strong 
parabolic maximum principle. The probabilistic viewpoint proves to be
particularly convenient for our analysis and will be adopted in the rest
of the paper.
\end{rem}

\section{The nondegenerate case}
\label{sec:nondeg}

As was explained in the previous section, the dichotomy between 
nondegenerate and degenerate support of the process $(X_t,Y_t)$ forms 
the basis of the different equality cases of the Ehrhard-Borell 
inequality. Cases of degenerate support can only arise, however, when 
the parabolic equation of Lemma \ref{lem:para} is degenerate, that is, 
$\rho=1$ or $\rho=-1$. In nondegenerate cases $-1<\rho<1$, the 
Laplacian $\Delta_\rho$ is uniformly elliptic and it is well understood 
that this implies that the law of $(X_t,Y_t)$ has a positive density 
with respect to Lebesgue measure (we include a short proof below for 
completeness). Thus, following the logic of the previous section, we 
expect in this case to obtain only the simple equality case 
\eqref{eq:H1} and its limiting case \eqref{eq:H2} when we admit 
non-smooth $f,g,h$. This is indeed precisely what happens.

Thus our first goal will be to settle the nondegenerate equality cases, 
which is the aim of this section. While this case is conceptually much 
simpler than the degenerate equality cases that will be studied in the 
next section, we must still address the regularity issues that we have 
ignored in our discussion so far. For concreteness, let us formulate the 
main result to be proved in this section.

\begin{prop}
\label{prop:nondeg}
Let $\lambda\ge\mu>0$ satisfy $\lambda+\mu>1$ and $\lambda-\mu<1$.
Let $f,g,h:\mathbb{R}^n\to[0,1]$ be nontrivial measurable
functions satisfying 
$$
	\Phi^{-1}(h(\lambda x+\mu y))
	\gea \lambda \Phi^{-1}(f(x)) + \mu \Phi^{-1}(g(y)).
$$
If equality holds in the Ehrhard-Borell inequality
$$
	\Phi^{-1}\bigg(\int h\, d\gamma_n\bigg) =
	\lambda\Phi^{-1}\bigg(\int f\, d\gamma_n\bigg) +
	\mu\Phi^{-1}\bigg(\int g\, d\gamma_n\bigg),
$$
then \textbf{either}
$$
	h(x) \eqa \Phi(\langle a,x\rangle + \lambda b+\mu c),\quad
	f(x) \eqa \Phi(\langle a,x\rangle + b),\quad
	g(x) \eqa \Phi(\langle a,x\rangle + c),
$$
\textbf{or}
$$
	h(x) \eqa 1_{\langle a,x\rangle + \lambda b+\mu c\ge 0},\quad
	f(x) \eqa 1_{\langle a,x\rangle + b\ge 0},\quad
	g(x) \eqa 1_{\langle a,x\rangle + c\ge 0},
$$
for some $a\in\mathbb{R}^n$ and $b,c\in\mathbb{R}$.
\end{prop}

The remainder of this section is devoted to the proof of this result. 
The assumptions of Proposition \ref{prop:nondeg} will be assumed to be 
in force throughout this section.

In the discussion of the previous section, we assumed for simplicity 
that $f,g,h$ are regular. To obtain the general equality cases, however, we 
cannot make this assumption. When we assumed no regularity on $f,g,h$, 
we cannot directly apply the representation of Lemma 
\ref{lem:fk} for $t=1$, and we must work with $t\in[0,1)$ only.
Let us begin by recalling a classic fact about 
nondegenerate diffusions.

\begin{lem}
\label{lem:girs}
In the present setting, the law of $(X_t,Y_t)$ has a positive density with 
respect to the Lebesgue measure on $\mathbb{R}^n\times\mathbb{R}^n$ for 
every $t\in(0,1)$.
\end{lem}

\begin{proof}
Let $Z_t$ be standard Brownian motion on $\mathbb{R}^{2n}$ and define
$$
	\Sigma := \left[
	\begin{matrix} I_n & \rho I_n \\ \rho I_n & I_n
	\end{matrix}
	\right],\qquad\quad
	\left[
	\begin{matrix}
	W_t\\
	B_t
	\end{matrix}
	\right] = 
	\Sigma^{1/2}Z_t,
$$
where $I_n$ denotes the $n\times n$ identity matrix.
Then $W_t,B_t$ are $\rho$-correlated $n$-dimensional standard
Brownian motions as defined in Lemma \ref{lem:fk}.

The assumptions $\lambda+\mu>1$ and $\lambda-\mu<1$ or, equivalently,
$-1<\rho<1$, imply that $\Sigma$ is nonsingular.
Let $\mathbf{P}_t$ be the law of $(X_s,Y_s)_{s\in[0,t]}$,
and define $\mathbf{Q}_t$ by
$$
	\frac{d\mathbf{Q}_t}{d\mathbf{P}_t} := 
	\exp\bigg(-\int_0^t \langle\Sigma^{-1/2}b(1-s,X_s,Y_s),dZ_s\rangle
	- \frac{1}{2}\int_0^t 
	\|\Sigma^{-1/2}b(1-s,X_s,Y_s)\|^2 ds\bigg).
$$
Then Girsanov's theorem \cite[Theorem 3.5.1]{KS88} states that the 
process $(X_s,Y_s)_{s\in[0,t]}$ has the same distribution under 
$\mathbf{Q}_t$ as does $(W_s,B_s)_{s\in[0,t]}$ under $\mathbf{P}_t$ (we 
showed in the proof of Lemma \ref{lem:fk} that the function $b$ has 
linear growth, so the assumption of Girsanov's theorem is satisfied by 
\cite[Corollary 3.5.16]{KS88}). It follows that $(X_t,Y_t)$ has a 
positive density with respect to a nondegenerate Gaussian measure on 
$\mathbb{R}^n\times\mathbb{R}^n$, and therefore with respect to Lebesgue 
measure, for every $t\in(0,1)$.
\end{proof}

Combining this observation with Lemma \ref{lem:fk}, we obtain the following.

\begin{cor}
\label{cor:ndeq}
In the present setting, if equality holds in the Ehrhard-Borell inequality,
then $C(t,x,y)=0$ for all $t\in(0,1)$ and $x,y\in\mathbb{R}^n$, that is,
$$
	\Phi^{-1}(Q_th(\lambda x+\mu y)) =
	\lambda\Phi^{-1}(Q_tf(x)) + \mu\Phi^{-1}(Q_tg(y)).
$$
\end{cor}

\begin{proof}
As $f,g,h$ satisfy the assumption of the Ehrhard-Borell inequality, so do
the functions
$\tilde f(z) := f(x+\sqrt{t}z)$,
$\tilde g(z) := g(y+\sqrt{t}z)$,
$\tilde h(z) := h(\lambda x+\mu y+\sqrt{t}z)$.
Applying the Ehrhard-Borell inequality (Theorem \ref{thm:be}) to the
latter functions shows that
$C(1-t,x,y)\ge 0$ for every $t\in(0,1)$, $x,y\in\mathbb{R}^n$. As we also
assumed that equality holds in the Ehrhard-Borell inequality for the functions
$f,g,h$, it follows from Lemma \ref{lem:fk} that $C(1-t,X_t,Y_t) = 0$ a.s.\ 
for every $t\in(0,1)$. In particular, Lemma \ref{lem:girs} implies that
$C(1-t,\cdot,\cdot)\eqa 0$. But Lemma \ref{lem:anal} implies that
$C(1-t,\cdot,\cdot)$ is continuous, so we can eliminate the null set to
obtain the conclusion.
\end{proof}

We now show that any three regular functions that satisfy the identity 
in Corollary \ref{cor:ndeq} must be linear. This explains the appearance 
of the equality case \eqref{eq:H1}.

\begin{lem}
\label{lem:eqcalc}
Let $u,v,w:\mathbb{R}^n\to\mathbb{R}$ be smooth functions such that
$$
	w(\lambda x+\mu y)=\lambda u(x)+\mu v(y)
$$
for all $x,y\in\mathbb{R}^n$. Then
$$
	w(x) = \langle a,x\rangle + \lambda b+\mu c,\qquad
	u(x) = \langle a,x\rangle + b,\qquad
	v(x) = \langle a,x\rangle + c
$$
for some $a\in\mathbb{R}^n$ and $b,c\in\mathbb{R}$.
\end{lem}

\begin{proof}
Differentiating the assumption with respect to $x_i$ or $y_i$ yields
$$
	\nabla w(\lambda x+\mu y) =
	\nabla u(x) = \nabla v(y)
$$
for all $x,y$. It follows that $\nabla w(x)=a$ must be a constant function,
and thus $\nabla u(x)=\nabla v(x)=a$ must equal the same constant.
This readily implies that
$$
	u(x) = \langle a,x\rangle + b,\qquad
	v(x) = \langle a,x\rangle + c,\qquad
	w(x) = \langle a,x\rangle + d,
$$
and plugging these forms into the assumption shows that
$d=\lambda b+\mu c$.
\end{proof}

Now recall that the functions $u_f,u_g,u_h$ are smooth for $t>0$ by 
Lemma \ref{lem:anal}. Combining Corollary \ref{cor:ndeq} and Lemma 
\ref{lem:eqcalc}, we have therefore shown that in the present setting, 
equality in the Ehrhard-Borell inequality implies that for every 
$t\in(0,1)$, there exist $a\in\mathbb{R}^n$ and $b,c\in\mathbb{R}$ such 
that
$$
	Q_th(x) = \Phi(\langle a,x\rangle+\lambda b+\mu c),\quad
	Q_tf(x) = \Phi(\langle a,x\rangle+b),\quad
	Q_tg(x) = \Phi(\langle a,x\rangle+c).
$$
To complete the proof of Proposition \ref{prop:nondeg}, it remains to 
invert the heat semigroup $Q_t$ to deduce the characterization of 
$f,g,h$. Fortunately, it is possible to do so as the heat semigroup 
$f\mapsto Q_tf$ is injective. This idea is made precise by the following 
lemma; an analogous argument can be found in the proof of 
\cite[Theorem 1]{CK01}.

\begin{lem}
\label{lem:invert}
Let $f:\mathbb{R}^n\to[0,1]$ be a measurable function such that
$Q_tf(x)=\Phi(\langle a,x\rangle+b)$ for some $t>0$, $a\in\mathbb{R}^n$, 
and $b\in\mathbb{R}$. Then $\|a\|\le t^{-1/2}$, and
$$
	f(x) \eqa \begin{cases}
	1_{\langle a,x\rangle + b\ge 0} & \mbox{if }\|a\|=t^{-1/2},\\
	\Phi\bigg(
	\frac{\langle a,x\rangle+b}{\sqrt{1-t\|a\|^2}}\bigg)
	& \mbox{if }\|a\|<t^{-1/2}.	
	\end{cases}
$$
\end{lem}

\begin{proof}
That $\|a\|\le t^{-1/2}$ follows as $\Phi^{-1}(Q_tf)$ is 
$t^{-1/2}$-Lipschitz by Lemma \ref{lem:anal}. Now note that if $f(x)$ 
has the form stated in the lemma, then indeed $Q_tf(x)=\Phi(\langle 
a,x\rangle+b)$ (this follows from the identities that appear in section 
\ref{sec:main} in the proof of sufficiency in Theorem \ref{thm:main}). 
It therefore suffices to show that the heat semigroup $Q_t$ is injective 
modulo null sets, that is, that $Q_tf=Q_tg$ if and only if $f\eqa g$. 

To this end, note that $Q_tf=f*\varphi_t$ with
$\varphi_t(x):=e^{-\|x\|^2/2t}/(2\pi t)^{n/2}$.
As $\varphi_t$ is a rapidly 
decreasing function whose Fourier transform is strictly positive, we 
have $Q_tf=Q_tg$ if and only if $\hat f=\hat g$ as tempered 
distributions \cite[Theorem 7.19]{Rud91}, which implies that $f=g$ a.e.\ by the 
Fourier inversion theorem \cite[Theorem 7.15]{Rud91}.
\end{proof}

Corollary \ref{cor:ndeq} and Lemmas \ref{lem:eqcalc} and 
\ref{lem:invert} complete the proof of Proposition \ref{prop:nondeg}.

\begin{rem}
In the setting of this section, Lemma \ref{lem:invert} is not really
essential. Indeed, Corollary \ref{cor:ndeq} and Lemma
\ref{lem:eqcalc} show that for every $t\in(0,1)$,
$$
	Q_th(x) = \Phi(\langle a_t,x\rangle+\lambda b_t+\mu c_t),\quad
	Q_tf(x) = \Phi(\langle a_t,x\rangle+b_t),\quad
	Q_tg(x) = \Phi(\langle a_t,x\rangle+c_t)
$$
for some $a_t,b_t,c_t$.
Thus the functions $f,g,h$ could be recovered by a simple limiting
argument letting $t\downarrow 0$. However, in the analysis of the
degenerate case in the next section, we will encounter a situation
where the characterization of $Q_tf,Q_tg,Q_th$ may be guaranteed to hold only
for \emph{some} $t>0$ rather than for \emph{every} $t>0$. In that setting,
the limiting argument is no longer available. The advantage of
Lemma \ref{lem:invert} is that it allows us to capture both situations
simultaneously.
\end{rem}

\section{The basic degenerate case}
\label{sec:basedeg}

Now that we have addressed the (easy) nondegenerate equality cases, we 
are ready to tackle the degenerate situation. The aim of this section is 
to execute the program outlined at the end of section \ref{sec:borell} 
for the degenerate case $\lambda+\mu=1$ in the simplest one-dimensional 
setting. The latter will simplify the analysis, and we extend the result 
to any dimension in the next section by an induction argument. The key 
difficulty of the problem arises already in one dimension, and the 
present section forms the core of our analysis. Our aim is to prove the 
following result.

\begin{prop}
\label{prop:deg1d}
Let $\lambda,\mu>0$ satisfy $\lambda+\mu=1$.
Let $f,g,h:\mathbb{R}\to[0,1]$ be nontrivial measurable
functions satisfying 
$$
	\Phi^{-1}(h(\lambda x+\mu y))
	\gea \lambda \Phi^{-1}(f(x)) + \mu \Phi^{-1}(g(y)).
$$
If equality holds in the Ehrhard-Borell inequality
$$
	\Phi^{-1}\bigg(\int h\, d\gamma_1\bigg) =
	\lambda\Phi^{-1}\bigg(\int f\, d\gamma_1\bigg) +
	\mu\Phi^{-1}\bigg(\int g\, d\gamma_1\bigg),
$$
then \textbf{either}
$$
	h(x) \eqa \Phi(ax + \lambda b+\mu c),\quad
	f(x) \eqa \Phi(ax + b),\quad
	g(x) \eqa \Phi(ax + c)
$$
for some $a,b,c\in\mathbb{R}$, \textbf{or}
$$
	h(x) \eqa 1_{ax + \lambda b+\mu c\ge 0},\quad
	f(x) \eqa 1_{ax + b\ge 0},\quad
	g(x) \eqa 1_{ax + c\ge 0}
$$
for some $a,b,c\in\mathbb{R}$, \textbf{or}
$$
	h(x) \eqa f(x) \eqa g(x) \eqa \Phi(V(x))
$$
for some concave function $V:\mathbb{R}\to\mathbb{\bar R}$.
\end{prop}

The remainder of this section is devoted to the proof of this result. 
The assumptions of Proposition \ref{prop:deg1d} will be assumed to be 
in force throughout this section.

\subsection{A geometric criterion for nondegenerate support}
\label{sec:lie}

The assumption of this section that $\lambda+\mu=1$ implies $\rho=1$. 
Therefore, in the present setting, $W_t$ is a one-dimensional standard 
Brownian motion and $B_t=W_t$ in Lemma \ref{lem:fk}. In particular, the 
two-dimensional stochastic differential equation for $(X_t,Y_t)$ is 
driven only by the one-dimensional Brownian motion $W_t$, that is,
\begin{align*}
	dX_t &= b_1(1-t,X_t,Y_t)\,dt + dW_t, \\
	dY_t &= b_2(1-t,X_t,Y_t)\,dt + dW_t,
\end{align*}
where we denote the first and second component of the function $b$
in Lemma \ref{lem:para} by $b_1,b_2$, respectively. To avoid
time-inhomogeneity, however, it will be convenient to view $(t,X_t,Y_t)$ as the
solution of a three-dimensional equation where $t$ has trivial dynamics.
The latter equation is time-homogeneous and
should be viewed as being driven by the
drift vector field $A$ and diffusion vector field $B$ on $\mathbb{R}^3$
defined by
$$
	A(t,x,y) := \frac{\partial}{\partial t} +
	b_1(1-t,x,y)\frac{\partial}{\partial x}+
	b_2(1-t,x,y)\frac{\partial}{\partial y},\qquad
	B(t,x,y) := \frac{\partial}{\partial x}+\frac{\partial}{\partial y}.
$$
By construction, these vector fields are smooth on $[0,1)\times\mathbb{R}^2$.

As was explained in section \ref{sec:borell}, at the heart of our proof 
lies the dichotomy between nondegenerate and degenerate support of
$(X_t,Y_t)$. In order to exploit this dichotomy, we will need a 
mechanism that translates the geometry of the support into an algebraic 
property of the function $b$, which will allow us to extract the 
equality cases. The aim of this subsection is to introduce the necessary 
machinery; the analysis of the equality cases will be done in the next 
subsection.

Before we give a precise formulation of the technique that we will use, 
let us provide a brief informal discussion. At every point $(t,x,y)$, 
our differential equation\footnote{One should 
	view this as an ordinary differential equation that is 
	forced by the vector field $A+B\,\dot W_t$: this idea will be made
	rigorous in Lemma \ref{lem:supp} below. More generally,
	stochastic differential equations of this kind may be defined
	in Stratonovich form \cite[Chapter V]{IW89}, but all subtleties
	of the general theory are avoided here as our vector field
	$B$ happens to be constant.
}
$$
	d(t,X_t,Y_t) = A(t,X_t,Y_t) + B(t,X_t,Y_t) 
	\,dW_t
$$
can push in any direction in the linear span of $A(t,x,y)$ and 
$B(t,x,y)$ as the magnitude of $dW_t$ is random. Now suppose $(X_t,Y_t)$ 
is supported at every time in a (possibly different) lower-dimensional 
submanifold of $\mathbb{R}^2$. Then 
the triple $(t,X_t,Y_t)$ must always lie in a lower-dimensional 
submanifold $M\subset\mathbb{R}^3$. Consequently, the linear span of 
$A(t,x,y)$ and $B(t,x,y)$ must be contained in the tangent space 
$T_{(t,x,y)}M$ for all $(t,x,y)\in M$, and thus every vector field in 
the Lie algebra generated by $A,B$ must do so as well (the latter is 
evident from the fact that the restrictions of $A,B$ to $M$ may be viewed 
intrinsically as vector fields on $M$, so their Lie brackets must live on 
$M$ as well). In particular, this Lie algebra cannot be 
full-dimensional at any such point. If we invert this logic, we expect 
that if the Lie algebra generated by $A,B$ is full-dimensional at some 
point in the support of $(X_t,Y_t)$, then the support cannot be 
degenerate. This is precisely the mechanism that we will exploit.

\begin{rem}
It should be evident from this informal discussion that an argument of 
this kind cannot establish global nondegeneracy of the support as we did
in Lemma \ref{lem:girs}; it can at best yield local nondegeneracy in a 
neighborhood of the point where the Lie algebra is full-dimensional.
This will suffice for our purposes, however, as analyticity of the
function $C$ (Lemma \ref{lem:anal}) implies that if equality $C=0$ holds 
in any open set, this must imply equality everywhere. This is one of a 
number of technical details of our program that will be developed in the 
proof below.
\end{rem}

We now proceed to make these informal ideas precise. For a
probability measure $\mu$ on a separable metric space $\Omega$,
the \emph{support} $\supp\mu$ is defined as
\begin{align*}
	\supp\mu := \mbox{}&
	\bigcap\{K\subseteq\Omega:K\mbox{ is closed, }\mu(K)=1\}
	\\
	= \mbox{}&
	\{x\in\Omega:\mu(V)>0\mbox{ for all open neighborhoods }V\ni x\}.
\end{align*}
That is, $\supp\mu$ is the smallest closed set of unit probability (that 
$\supp\mu$ itself has unit probability follows as every open cover on 
$\Omega$ has a countable subcover \cite[\S II.2]{Par67}). In the sequel, 
we will denote by $\mu_t$ the distribution of $(X_t,Y_t)$ on 
$\mathbb{R}^2$.

\begin{prop}
\label{prop:lie}
Let $\mathscr{L}$ be the Lie algebra generated by $A,B$, that is, the
linear span of the vector fields $A,B,[A,B],[A,[A,B]],[B,[A,B]],\ldots$
Suppose that there exists a time $t\in(0,1)$ and 
a point $(x,y)\in\supp\mu_t$ for which
$$
	\dim(\{F(t,x,y):F\in\mathscr{L}\}) = 3.
$$
Then $\supp\mu_t$ contains a (non-empty) open subset of $\mathbb{R}^2$.
\end{prop}

This result provides the key tool that will be used in the following 
subsection to characterize the equality cases. The rest of this 
subsection is devoted to its proof.

The proof of Proposition \ref{prop:lie} is not really new, but rather 
combines two classical results: the Stroock-Varadhan support theorem 
\cite{SV72} and the characterization of local accessibility in nonlinear 
control theory \cite[\S 3.1]{NvdS16}. Unfortunately, the result of 
\cite{SV72} cannot be applied directly in the present setting as it 
requires the coefficients of the underlying stochastic differential 
equation to be bounded. While this could be addressed by an additional 
localization argument, the present setting proves to be particularly 
simple due to the fact that the diffusion vector field $B$ is constant, 
so we find it easier and more illuminating to give a direct proof.

Before we proceed, let us first record a simple and well-known observation.

\begin{lem}
\label{lem:imgsupp}
Let $\Omega,\Omega'$ be separable metric spaces, $\mu$ a probability
measure on $\Omega$, and $\iota:\Omega\to\Omega'$ a continuous function.
Then $\supp(\mu\circ\iota^{-1}) = \mathop{\mathrm{cl}}\iota(\supp\mu)$.
\end{lem}

\begin{proof}
Let $\nu=\mu\circ\iota^{-1}$. If
$\omega\in\supp\mu$, then for any open neighborhood $V$ of $\iota(\omega)$,
we have $\nu(V)=\mu(\iota^{-1}(V))>0$ as $\iota^{-1}(V)$ is an open 
neighborhood of $\omega$. Thus $\iota(\supp\mu)\subseteq 
\supp\nu$, and as $\supp\nu$ is closed we obtain
$\mathop{\mathrm{cl}}\iota(\supp\mu)\subseteq\supp\nu$.
But note that $\nu(\mathop{\mathrm{cl}}\iota(\supp\mu))
\ge\mu(\supp\mu)=1$. As $\supp\nu$ is the smallest closed set with
unit measure, we have also shown the converse inclusion $\supp\nu\subseteq 
\mathop{\mathrm{cl}}\iota(\supp\mu)$.
\end{proof}

We are now ready to state the first ingredient of the proof of 
Proposition \ref{prop:lie}. In the discussion at the beginning of this 
subsection, we stated that our stochastic differential equation ``can 
push in any direction in the linear span of $A$ and $B$.'' This 
intuitively obvious statement is made precise by the following result, 
which is a form of the Stroock-Varadhan support theorem. It states that 
the support of the law of $(X_s,Y_s)_{s\in[0,t]}$, viewed as a random 
continuous path, is precisely the set of all solutions of ordinary 
differential equations that are driven at every time $t$ by any vector 
field of the form $A+\dot h(t)B$ (provided $\dot h$ is measurable).

\begin{lem}
\label{lem:supp}
Let $t\in(0,1)$, and denote by $\mathcal{C}_0([0,t];\mathbb{R}^n)$ the set 
of continuous paths $h:[0,t]\to\mathbb{R}^n$ with $h(0)=0$, endowed with 
the topology of uniform convergence.
For any $h\in\mathcal{C}_0([0,t];\mathbb{R})$, let $(x^h(s),y^h(s))$ be
the solution of the differential equation
\begin{align*}
	x^h(s) &= \int_0^s b_1(1-u,x^h(u),y^h(u))\,du + h(s), \\
	y^h(s) &= \int_0^s b_2(1-u,x^h(u),y^h(u))\,du + h(s).
\end{align*}
Define the measure $\mathbf{P}_t$ on 
$\mathcal{C}_0([0,t];\mathbb{R}^2)$ to be the law of $(X_s,Y_s)_{s\in[0,t]}$. Then
$$
	\supp\mathbf{P}_t = \mathop{\mathrm{cl}}(\{
	(x^h(s),y^h(s))_{s\in[0,t]}:h\in \mathcal{C}_0([0,t];\mathbb{R})\}).
$$
\end{lem}

\begin{proof}
Define the measure $\mathbf{W}_t$ on $\mathcal{C}_0([0,t];\mathbb{R})$ to
be the law of the Brownian motion $W_{[0,t]}:=(W_s)_{s\in[0,t]}$. It is a classical
fact that $\supp\mathbf{W}_t = \mathcal{C}_0([0,t];\mathbb{R})$. Indeed,
choose any $\omega\in\supp\mathbf{W}_t$, so that
$\mathbf{P}[\|W_{[0,t]}-\omega\|_\infty<\varepsilon]>0$ for all
$\varepsilon>0$. Then by Girsanov's theorem \cite[Theorem 3.5.1]{KS88},
we also have
$\mathbf{P}[\|W_{[0,t]}-\omega-h\|_\infty<\varepsilon]>0$ for all
$\varepsilon>0$ whenever $h\in\mathcal{C}_0([0,t];\mathbb{R})$ satisfies
$\int_0^t \big|\frac{dh}{ds}\big|^2\,ds<\infty$. Thus $\omega+h\in\supp\mathbf{W}_t$
for any such $h$. As such $h$ are dense in $\mathcal{C}_0([0,t];\mathbb{R})$,
the claim follows.

Now recall that the function $b$ is locally Lipschitz and satisfies a
linear growth condition, as was shown in the proof of Lemma \ref{lem:fk}.
Therefore, by a standard existence and uniqueness argument, the ordinary
differential equation for $(x^h(s),y^h(s))$ has a unique solution.
Thus the map $\iota:h\mapsto(x^h(s),y^h(s))_{s\in[0,t]}$ is well defined
and $\mathbf{P}_t=\mathbf{W}_t\circ\iota^{-1}$. By
Lemma \ref{lem:imgsupp}, it remains to show that the map
$\iota$ is continuous.

The latter is readily established as follows. Suppose first that $b$ is
globally Lipschitz with constant $L$. Then for any $g,h\in\mathcal{C}_0([0,t];
\mathbb{R})$ and $s\in[0,t]$, we have
$$
	\delta_{g,h}(s) \le
	2L\int_0^s \delta_{g,h}(u)\,du + 2\|g-h\|_\infty
$$
where $\delta_{g,h}(s):= |x^g(s)-x^h(s)|+|y^g(s)-y^h(s)|$. Thus
$\iota$ is Lipschitz, as by Gr\"onwall's lemma $\|\iota(g)-\iota(h)\|_\infty\le
2e^{2Lt}\|g-h\|_\infty$. Now suppose $b$ is only locally Lipschitz
and that $\|h\|_\infty\le K$. Then using the linear growth 
condition, we can estimate
$$
	1+|x^h(s)|+|y^h(s)| \lesssim
	\int_0^s \{1+|x^h(u)|+|y^h(s)|\}\,du + 1+K,
$$
so that $\|x^h\|_\infty + \|y^h\|_\infty \lesssim 1+K$ by Gr\"onwall.
In particular, $(x^h,y^h)$ never leaves a ball of radius
$C(1+K)$ for a universal constant $C$. We can
therefore modify $b$ outside this ball to make it globally Lipschitz
without changing the solution, where the Lipschitz constant depends
on $K$ only. Combining this observation with the Lipschitz bound above
shows that $\iota$ is locally Lipschitz, and hence continuous.
\end{proof}

The ordinary differential equation of Lemma \ref{lem:supp} 
may be viewed as a control system: a controller may choose the input 
$h$ of the equation to guide the dynamics to a desired location. The set
of all locations that can be reached by time $t$
$$
	R(t):=\{(x^h(t),y^h(t)):h\in\mathcal{C}_0([0,t],\mathbb{R})\}
	\subseteq \mathbb{R}^2
$$
is called the \emph{reachable set} of the control system. Lemmas
\ref{lem:imgsupp} and \ref{lem:supp} imply that
$$
	\supp\mu_t = \mathop{\mathrm{cl}}R(t).
$$
Therefore, in order to prove Proposition \ref{prop:lie}, it suffices
to show that $R(t)$ contains an open set. This will be accomplished using
the following classical result from
nonlinear control theory \cite[Theorem 3.21 and Theorem 3.9]{NvdS16}.

\begin{lem}
\label{lem:reach}
Let $\mathscr{L}$ be the Lie algebra generated by $A,B$, that is, the
linear span of the vector fields $A,B,[A,B],[A,[A,B]],[B,[A,B]],\ldots$
Suppose that there exists a time $t\in[0,1)$ and 
a point $(x,y)\in R(t)$ for which
$$
	\dim(\{F(t,x,y):F\in\mathscr{L}\}) = 3.
$$
Then $R(t+\varepsilon)$ contains a (non-empty) open set for every $\varepsilon>0$.
\end{lem}

\begin{proof}
The proof is identical to that of \cite[Theorem 3.21]{NvdS16} (the result
is formulated there for time-homogeneous vector fields, but the proof 
applies verbatim in the present setting). We omit the details, which
essentially amount to a careful implementation of the ideas described
at the beginning of this subsection.
\end{proof}

Combining the above results, we readily complete the
proof of Proposition \ref{prop:lie}.

\begin{proof}[Proof of Proposition \ref{prop:lie}]
Fix $t\in(0,1)$ and $(x,y)\in\supp\mu_t$ such that $\mathscr{L}$
has full dimension at $(t,x,y)$. As the vector fields $A,B$ are 
smooth, it follows that $\mathscr{L}$ has full dimension in an
open neighborhood of $(t,x,y)$. In particular, as
$\supp\mu_t=\mathop{\mathrm{cl}}R(t)$ by Lemma \ref{lem:supp},  
we can find $t'<t$ and $(x',y')\in R(t')$ such that $\mathscr{L}$
has full dimension at $(t',x',y')$. Thus $R(t)\subseteq\supp\mu_t$
contains a non-empty open set by Lemma \ref{lem:reach}.
\end{proof}

\begin{rem}
The connection between the support of diffusion processes and methods of 
nonlinear control theory has long been known, cf.\ 
\cite{Kun78}. What is not so obvious is that this mechanism enables 
characterization of the equality cases in geometric inequalities, which
will be shown presently. Let us also note that the Lie-algebraic condition
of Proposition \ref{prop:lie} is strongly reminiscent of the H\"ormander
condition for hypoellipticity, and indeed the same conclusion could be deduced
by applying a variant of this machinery (see, for example, \cite{HLT17} and
the references therein). However, hypoellipticity is a much stronger 
condition than is needed for our purposes, and the present direct
approach could prove to be more flexible in other situations.
\end{rem}

\subsection{Analysis of the equality cases}

With Proposition \ref{prop:lie} in hand, we are finally ready to proceed 
to the main argument. We begin by computing an explicit form of the 
assumption of Proposition \ref{prop:lie} for our particular choice of 
vector fields $A,B$.

\begin{lem}
\label{lem:exlie}
Suppose there exists $t\in(0,1)$ and $(x,y)\in\supp\mu_t$ such that
$$
	\frac{\partial b_1(1-t,x,y)}{\partial x}+
	\frac{\partial b_1(1-t,x,y)}{\partial y} \ne
	\frac{\partial b_2(1-t,x,y)}{\partial x}+
	\frac{\partial b_2(1-t,x,y)}{\partial y}.
$$
Then the condition of Proposition \ref{prop:lie} holds.
\end{lem}

\begin{proof}
An easy computation shows that
\begin{align*}
	[B,A](t,x,y) &=
	\left(\frac{\partial b_1(1-t,x,y)}{\partial x}+
	\frac{\partial b_1(1-t,x,y)}{\partial y}\right)\frac{\partial}{\partial x}
	\\ &\qquad +
	\left(\frac{\partial b_2(1-t,x,y)}{\partial x}+
	\frac{\partial b_2(1-t,x,y)}{\partial y}\right)\frac{\partial}{\partial y}.	
\end{align*}
We trivially have $A(t,x,y)\not\in \mathrm{span}\{B(t,x,y),[B,A](t,x,y)\}$
as $A$ is the only vector field with a component in the time direction. The
assumption of the lemma further implies that $B(t,x,y)$ and $[B,A](t,x,y)$
are linearly independent. Hence the linear span of 
$A(t,x,y),B(t,x,y),[B,A](t,x,y)$ is full-dimensional.
\end{proof}

Roughly speaking, the proof of Proposition \ref{prop:deg1d} will 
proceed along the following lines. Suppose first that the assumption of 
Lemma \ref{lem:exlie} is satisfied. Then Proposition \ref{prop:lie} yields 
nondegeneracy of the support, and we may argue essentially as we did in 
Corollary \ref{cor:ndeq} to obtain the equality cases (see Lemma 
\ref{lem:degndeg} below).

On the other hand, suppose the assumption of Lemma \ref{lem:exlie} is 
violated. Then, by assumption, we obtain for every $t\in(0,1)$ and 
$(x,y)\in\supp\mu_t$ an identity between the derivatives of $b_1$ and 
$b_2$. One may hope that if one substitutes into this identity the 
explicit expressions for $b_1,b_2$ given in Lemma \ref{lem:para}, then the 
desired equality cases will follow from the resulting equation. 
Unfortunately, this single equation does not suffice to obtain the 
equality cases. However, we have not exploited so far a very useful fact: 
as $C(t,x,y)\ge 0$ everywhere and $C(t,x,y)=0$ for $(x,y)\in\supp\mu_t$, 
the latter points are all minimizers of the function $C$. Therefore, the 
first- and second-order conditions of optimality provide us with 
additional differential identities that must be satisfied on $\supp\mu_t$, 
which are recorded in the following lemma.

\begin{lem}
\label{lem:exmin}
Suppose that equality holds in the Ehrhard-Borell inequality. Then
\begin{align*}
	&
	\frac{\partial C(1-t,x,y)}{\partial x} =
	\frac{\partial C(1-t,x,y)}{\partial y} = 0, \\
	&\frac{\partial^2}{\partial x^2}C(1-t,x,y) =
	\frac{\partial^2}{\partial y^2}C(1-t,x,y) =
	-\frac{\partial^2}{\partial x\partial y}C(1-t,x,y)
\end{align*}
for every $t\in(0,1)$ and $(x,y)\in\supp\mu_t$.
\end{lem}

\begin{proof}
Arguing exactly as in the proof of Corollary \ref{cor:ndeq}, we find 
that $C(1-t,x,y)\ge 0$ for all $(t,x,y)\in[0,1)\times\mathbb{R}^2$, while
$C(1-t,x,y)=0$ whenever $t\in(0,1)$ and $(x,y)\in\supp\mu_t$. In 
particular, every point of the latter type is a minimizer of the smooth
function $C$. Consequently, the first-order conditions for optimality yield
$$
	\frac{\partial C(1-t,x,y)}{\partial t} =
	\frac{\partial C(1-t,x,y)}{\partial x} =
	\frac{\partial C(1-t,x,y)}{\partial y} = 0,
$$
which implies by Lemma \ref{lem:para} that
$$
	\Delta_\rho C(1-t,x,y) =
	\frac{\partial^2 C(1-t,x,y)}{\partial x^2} +
	2\frac{\partial^2 C(1-t,x,y)}{\partial x\partial y} +
	\frac{\partial^2 C(1-t,x,y)}{\partial y^2} = 0.
$$
On the other hand, by the second-order condition for optimality, the
Hessian of $C(1-t,\cdot,\cdot)$ is positive semidefinite, and thus its
determinant is nonnegative
$$
	\frac{\partial^2 C(1-t,x,y)}{\partial x^2} 
	\frac{\partial^2 C(1-t,x,y)}{\partial y^2} -
	\bigg(
	\frac{\partial^2 C(1-t,x,y)}{\partial x\partial y}
	\bigg)^2\ge 0.
$$
Combining the last two identities yields
$$
	-\frac{1}{4}\left(
	\frac{\partial^2 C(1-t,x,y)}{\partial x^2}- 
	\frac{\partial^2 C(1-t,x,y)}{\partial y^2}	
	\right)^2\ge 0,
$$
which implies
$$
	\frac{\partial^2 C(1-t,x,y)}{\partial x^2} =
	\frac{\partial^2 C(1-t,x,y)}{\partial y^2}.
$$
The remaining conclusion follows by using again $\Delta_\rho C=0$.
\end{proof}

When taken together, the differential identities thus obtained will turn 
out to suffice to deduce (by means of careful explicit computations) all 
the equality cases. Let us summarize this program by stating the 
key \textbf{dichotomy}:
\begin{enumerate}[$\bullet$]
\itemsep\abovedisplayskip
\item \textbf{Nondegenerate case:}
there exists $t\in(0,1)$ and $(x,y)\in\supp\mu_t$ such that
\begin{equation}
\label{eq:N}
\tag{N}
	\frac{\partial b_1(1-t,x,y)}{\partial x}+
	\frac{\partial b_1(1-t,x,y)}{\partial y} \ne
	\frac{\partial b_2(1-t,x,y)}{\partial x}+
	\frac{\partial b_2(1-t,x,y)}{\partial y}.
\end{equation}
In this case, Proposition \ref{prop:lie} will allow us to argue precisely
as in section \ref{sec:nondeg} that equality in the Ehrhard-Borell inequality
implies \eqref{eq:H1} or \eqref{eq:H2}.
\item \textbf{Degenerate case:} for all
$t\in(0,1)$ and $(x,y)\in\supp\mu_t$
\begin{equation}
\label{eq:D}
\tag{D}
\begin{aligned}
	&\frac{\partial b_1(1-t,x,y)}{\partial x}+
	\frac{\partial b_1(1-t,x,y)}{\partial y} =
	\frac{\partial b_2(1-t,x,y)}{\partial x}+
	\frac{\partial b_2(1-t,x,y)}{\partial y}, \\
	&\frac{\partial C(1-t,x,y)}{\partial x} =
	\frac{\partial C(1-t,x,y)}{\partial y} = 0,\\
	&\frac{\partial^2}{\partial x^2}C(1-t,x,y) =
	\frac{\partial^2}{\partial y^2}C(1-t,x,y) =
	-\frac{\partial^2}{\partial x\partial y}C(1-t,x,y).
\end{aligned}
\end{equation}
These identities place strong constraints on the possible behavior of the
functions $u_f,u_g,u_h$. The bulk of the analysis lies in this case: we
will show that these identities imply that one of the equality cases in 
Proposition \ref{prop:deg1d} must hold.
\end{enumerate}
\vskip.1cm

\subsubsection{The nondegenerate case}

This case is captured by the following lemma.

\begin{lem}
\label{lem:degndeg}
Suppose that the nondegenerate case \eqref{eq:N} is in force. If equality holds
in the Ehrhard-Borell inequality, then either
$$
	h(x) \eqa \Phi(ax + \lambda b+\mu c),\quad
	f(x) \eqa \Phi(ax + b),\quad
	g(x) \eqa \Phi(ax + c),
$$
or
$$
	h(x) \eqa 1_{ax + \lambda b+\mu c\ge 0},\quad
	f(x) \eqa 1_{ax + b\ge 0},\quad
	g(x) \eqa 1_{ax + c\ge 0},
$$
for some $a,b,c\in\mathbb{R}$.
\end{lem}

\begin{proof}
Arguing exactly as in the proof of Corollary \ref{cor:ndeq}, we
have $C(1-t,x,y)=0$ for all $t\in(0,1)$ and $(x,y)\in\supp\mu_t$.
By Proposition \ref{prop:lie} and Lemma \ref{lem:exlie}, the nondegeneracy
assumption \eqref{eq:N} implies that $\supp\mu_t$ contains a non-empty open
set for some $t\in(0,1)$. For this $t$, the function $C(1-t,\cdot,\cdot)$
vanishes on a non-empty open set in $\mathbb{R}^2$, and as this function is
analytic by Lemma \ref{lem:anal} it must vanish everywhere on $\mathbb{R}^2$.
The proof is concluded by applying 
Lemmas \ref{lem:eqcalc} and \ref{lem:invert}.
\end{proof}

\subsubsection{The degenerate case}

This case is more involved. Let us begin by computing an explicit form
of the identities \eqref{eq:D} using the definitions of the functions
$b_1,b_2,C$. Remarkably, this computation results in a further dichotomy.

\begin{lem}
\label{lem:degdegd}
Suppose that the degenerate case \eqref{eq:D} is in force. If equality holds
in the Ehrhard-Borell inequality, then for every
$t\in(0,1)$ and $(x,y)\in\supp\mu_t$, either
\begin{equation}
\label{eq:D1}
\tag{D1}
	u_f(1-t,x)=u_g(1-t,y)
\end{equation}
or
\begin{equation}
\label{eq:D2}
\tag{D2}
	u_h''(1-t,\lambda x+\mu y) = u_f''(1-t,x) = u_g''(1-t,y) = 0,
\end{equation}
where $u_h''(t,x):=\frac{\partial^2}{\partial x^2}u_h(t,x)$ and
analogously for $u_f,u_g$.
\end{lem}

\begin{proof}
Substituting the definitions of $b(t,x,y)$ and $C(t,x,y)$
(cf.\ section \ref{sec:borell}) into the identities
contained in \eqref{eq:D} yields, respectively, that
(here $u_h'(t,x):=\frac{\partial}{\partial x}u_h(t,x)$)
\begin{align*}
	&u_f'(1-t,x)(u_h'(1-t,\lambda x+\mu y)+u_f'(1-t,x)) \\
	&\qquad\quad+
	u_f(1-t,x)(u_h''(1-t,\lambda x+\mu y)+u_f''(1-t,x)) =\mbox{}
	\\ &
	u_g'(1-t,y)(u_h'(1-t,\lambda x+\mu y)+u_g'(1-t,y)) \\
	&\qquad\quad+
	u_g(1-t,y)(u_h''(1-t,\lambda x+\mu y)+u_g''(1-t,y))
\end{align*}
and
\begin{align*}
	u_h'(1-t,\lambda x+\mu y) &=u_f'(1-t,x)=u_g'(1-t,y),\\
	u_h''(1-t,\lambda x+\mu y) &= u_f''(1-t,x) = u_g''(1-t,y)
\end{align*}
for every $t\in(0,1)$ and $(x,y)\in\supp\mu_t$. Combining these equations
gives
$$
	(u_f(1-t,x)-u_g(1-t,y)) u_h''(1-t,\lambda x+\mu y) = 0
$$
for every $t\in(0,1)$ and $(x,y)\in\supp\mu_t$. The conclusion follows.
\end{proof}

We will consider two separate situations:
\begin{enumerate}[$\bullet$]
\itemsep\abovedisplayskip
\item \textbf{Case 1:} \eqref{eq:D1} holds for all $t\in(0,1)$ and
$(x,y)\in\supp\mu_t$. In this case, we will show that the third equality
case of Proposition \ref{prop:deg1d} must hold.
\item \textbf{Case 2:} There exists $t\in(0,1)$ and $(x,y)\in\supp\mu_t$
such that \eqref{eq:D1} does not hold. In this case, we will obtain the same
equality cases as in Lemma \ref{lem:degndeg}.
\end{enumerate}
\vskip.1cm

Let us begin by analyzing the first case. 

\begin{lem}
\label{lem:concave}
Suppose that \eqref{eq:D1} holds for all $t\in(0,1)$ and
$(x,y)\in\supp\mu_t$. If equality holds in the Ehrhard-Borell inequality,
then we have
$$
	h(x)\eqa f(x)\eqa g(x)\eqa \Phi(V(x))
$$
for some concave function $V:\mathbb{R}\to\mathbb{\bar R}$.
\end{lem}

\begin{proof}
The assumption states that
$$
	u_f(1-t,x)=u_g(1-t,y),\qquad
	u_h'(1-t,\lambda x+\mu y)=u_f'(1-t,x)=u_g'(1-t,y)
$$
for every $t\in(0,1)$ and $(x,y)\in\supp\mu_t$, where the second identity
was shown in the proof of Lemma \ref{lem:degdegd}. In particular, this
implies that
$$
	u_f(1-t,X_t)=u_g(1-t,Y_t),~~~
	u_h'(1-t,\lambda X_t+\mu Y_t) = 
	u_f'(1-t,X_t)=u_g'(1-t,Y_t)~~~
	\mbox{a.s.}
$$
for every $t\in(0,1)$. Therefore, the definition of $b$ in
Lemma \ref{lem:para} yields
$$
	b_1(1-t,X_t,Y_t) = b_2(1-t,X_t,Y_t)\quad\mbox{a.s.}
$$
Thus the equation for $X_t,Y_t$ 
at the beginning of section \ref{sec:lie} and $X_0=Y_0=0$ yield
$$
	X_t=Y_t\quad\mbox{a.s.}\quad\mbox{for every }t\in [0,1),
$$
so that $\supp\mu_t\subseteq\{(x,x):x\in\mathbb{R}\}$. Using Lemma
\ref{lem:imgsupp}, we can now conclude that
$\supp\mu_t=\{(x,x):x\in\supp\mu_t^X\}$ where $\mu_t^X$
denotes the law of $X_t$.

In fact, we claim that $\supp\mu_t^X=\mathbb{R}$, so that
$$
	\supp\mu_t = \{(x,x):x\in\mathbb{R}\}.
$$
To see why, note that by the same reasoning as above, we have
$b_1(1-t,X_t,Y_t)=-u_f(1-t,X_t)u_f'(1-t,X_t)$ a.s., so that $X_t$ satisfies
the equation
$$
	dX_t = -u_f(1-t,X_t)u_f'(1-t,X_t)\,dt + dW_t.
$$
The claim follows by the identical argument as in the proof of
Lemma \ref{lem:girs}.

With the characterization of $\supp\mu_t$ in hand, it follows immediately
using the definitions of $u_f,u_g$ that \eqref{eq:D1} implies 
$Q_{1-t}f = Q_{1-t}g$ for all $t\in(0,1)$.
Now recall from the proof of 
Corollary \ref{cor:ndeq} that $C\ge 0$ everywhere, which implies in the present 
case that $C(1-t,x,x)=\Phi^{-1}(Q_{1-t}h(x))-\Phi^{-1}(Q_{1-t}f(x))\ge 0$.
On the other hand, using equality in the
Ehrhard-Borell inequality we obtain
\begin{align*}
	\Phi^{-1}\bigg(\int Q_{1-t}h(\sqrt{t}x)\,\gamma_1(dx)\bigg) &=
	\Phi^{-1}\bigg(
	\int h\,d\gamma_1\bigg) \\ &=
	\lambda
	\Phi^{-1}\bigg(
	\int f\,d\gamma_1\bigg) +
	\mu
	\Phi^{-1}\bigg(
	\int g\,d\gamma_1\bigg) \\ &=
	\Phi^{-1}\bigg(
	\int Q_{1-t}f(\sqrt{t}x)\,\gamma_1(dx)
	\bigg),
\end{align*}
where we used $Q_{1-t}f=Q_{1-t}g$ in the last line. Thus we have shown
$$
	Q_{1-t}h\ge Q_{1-t}f,\qquad
	\int Q_{1-t}h(\sqrt{t}x)\,\gamma_1(dx) =	
	\int Q_{1-t}f(\sqrt{t}x)\,\gamma_1(dx).
$$
It follows that
$Q_{1-t}h\eqa Q_{1-t}f$, and as these functions are smooth we have shown
$$
	Q_{1-t}h=Q_{1-t}f=Q_{1-t}g\quad\mbox{for all }t\in(0,1).
$$
Thus
$$
	h(x)\eqa f(x)\eqa g(x)
$$
follows by injectivity of the heat semigroup as in
the proof of Lemma \ref{lem:invert}.

It remains to argue that $f(x)\eqa\Phi(V(x))$ for some concave function
$V:\mathbb{R}\to\mathbb{\bar R}$. To this end, we use again
$C\ge 0$ and $Q_{1-t}h=Q_{1-t}f=Q_{1-t}g$ to conclude that
$$
	\Phi^{-1}(Q_{1-t}f(\lambda x+\mu y)) \ge
	\lambda\Phi^{-1}(Q_{1-t}f(x))+
	\mu\Phi^{-1}(Q_{1-t}f(y))
$$
for all $t\in(0,1)$ and $x,y\in\mathbb{R}$. As $Q_{1-t}f$ is smooth and
$\lambda+\mu=1$, this implies that the function $\Phi^{-1}(Q_{1-t}f)$ 
is concave for every $t\in(0,1)$. We claim that
$$
	V(x) := \liminf_{t\to 1} \Phi^{-1}(Q_{1-t}f(x))
$$
satisfies the requisite properties. Indeed, using that $Q_{1-t}f(x)\to f(x)$
a.e.\ as $t\to 1$ \cite[Theorem 8.15]{Fol99}, we clearly have
$f(x)\eqa \Phi(V(x))$. On the other hand,
\begin{align*}
	V(\alpha x+(1-\alpha) y) &=
	\liminf_{t\to 1} \Phi^{-1}(Q_{1-t}f(\alpha x+(1-\alpha)y)) \\
	&\ge
	\liminf_{t\to 1} \{\alpha \Phi^{-1}(Q_{1-t}f(x)) +
	(1-\alpha)\Phi^{-1}(Q_{1-t}f(y))\} \\
	&\ge \alpha V(x) + (1-\alpha)V(y)
\end{align*}
for every $\alpha\in[0,1]$ and $x,y\in\mathbb{R}$, where we used that
$\Phi^{-1}(Q_{1-t}f)$ is concave for $t\in(0,1)$. Thus $V$ is concave,
and we have completed the proof.
\end{proof}

We finally tackle the last remaining case. Before we turn to the proof, we 
record an elementary property of real analytic functions on 
$\mathbb{R}$ that will be used below.

\begin{lem}
\label{lem:accum}
Let $u:\mathbb{R}\to\mathbb{R}$ be real analytic, and let $(x_k)_{k\ge 0}
\subset\mathbb{R}$ satisfy $x_k\to x$ and $x_k\ne x$ for all $k$.
If $u(x_k)=0$ for all $k$, then $u(z)=0$ for all $z\in\mathbb{R}$.
\end{lem}

\begin{proof}
Write $u(z)=\sum_{n\ge 0}a_n(z-x)^n$ as an absolutely convergent power series.
Clearly $a_0=0$. Now suppose we have shown $a_0,\ldots,a_{n-1}=0$. Then
$$
	0 = \frac{u(x_k)}{(x_k-x)^n} =
	\sum_{m\ge n} a_m (x_k-x)^{m-n} \to
	a_n\quad\mbox{as }k\to\infty.
$$
Thus the conclusion follows by induction.
\end{proof}

We are now ready to complete the proof of Proposition \ref{prop:deg1d}.

\begin{lem}
Suppose there exist $t\in(0,1)$ and $(x,y)\in\supp\mu_t$ 
such that \eqref{eq:D1} does not hold. If equality holds in the Ehrhard-Borell 
inequality, then either
$$
	h(x) \eqa \Phi(ax + \lambda b+\mu c),\quad
	f(x) \eqa \Phi(ax + b),\quad
	g(x) \eqa \Phi(ax + c),
$$
for all $x\in\mathbb{R}$, or
$$
	h(x) \eqa 1_{ax + \lambda b+\mu c\ge 0},\quad
	f(x) \eqa 1_{ax + b\ge 0},\quad
	g(x) \eqa 1_{ax + c\ge 0},
$$
for all $x\in\mathbb{R}$, for some $a,b,c\in\mathbb{R}$.
\end{lem}

\begin{proof}
Fix throughout the proof $t\in(0,1)$ and $(x,y)\in\supp\mu_t$ at which
\eqref{eq:D1} fails. We begin by showing that there exists a sequence
of points $(x_k,y_k)\in\supp\mu_t$ such that $x_k\to x$, $y_k\to y$,
and $x_k\ne x$, $y_k\ne y$, $\lambda x_k+\mu y_k\ne \lambda x+\mu y$
for all $k$. Indeed, suppose this is false.
Then there exists a neighborhood $V$ of $(x,y)$ such that every point
$(x',y')\in V$ satisfies either $x'=x$, or $y'=y$, or $\lambda x'+\mu y'=
\lambda x+\mu y$. In particular, as
$\mu_t(V)>0$ by definition, this would imply
$$
	\mathbf{P}[X_t=x\mbox{ or }Y_t=y\mbox{ or }
	\lambda X_t+\mu Y_t=\lambda x+\mu y
	|(X_t,Y_t)\in V]
	\stackrel{?}{=}1.
$$
However, the marginal laws of the random variables $X_t$, $Y_t$, and 
$\lambda X_t+\mu Y_t$ are all absolutely continuous with respect to 
Lebesgue measure, as each can be written by their definition in Lemma 
\ref{lem:fk} as a standard one-dimensional Brownian motion plus drift 
\cite[Theorem 7.2]{LS01}. It follows that $\mathbf{P}[X_t=x\mbox{ or 
}Y_t\mbox{ or }\lambda X_t+\mu Y_t=\lambda x+\mu y]=0$, and we 
have therefore obtained the desired contradiction.

Now fix a sequence $(x_k,y_k)\in\supp\mu_t$ with the above property.
If \eqref{eq:D1} were to hold infinitely often on this sequence, then
by continuity \eqref{eq:D1} must hold at $(x,y)$ which contradicts the
assumption. Thus Lemma \ref{lem:degdegd} ensures that \eqref{eq:D2} holds
eventually on this sequence. But this implies in particular, using
Lemma \ref{lem:accum}, that
$$
	u_h''(1-t,z)=
	u_f''(1-t,z)=u_g''(1-t,z)=0
$$
for all $z\in\mathbb{R}$, where we used that $u_f,u_g,u_h$ are
analytic (Lemma \ref{lem:anal}). Therefore
$$
	Q_th(z) = \Phi(a_1z+b_1),\qquad
	Q_tf(z) = \Phi(a_2z+b_2),\qquad
	Q_tg(z) = \Phi(a_3z+b_3)
$$
for some $a_1,a_2,a_3,b_1,b_2,b_3\in\mathbb{R}$.
But we have shown in the proof of Lemma
\ref{lem:degdegd} that 
$u_h'(1-t,\lambda x+\mu y)=u_f'(1-t,x)=u_g'(1-t,y)$ for $(x,y)\in\supp\mu_t$,
so that we must have $a_1=a_2=a_3$. Applying Lemma \ref{lem:invert}, we
find that either
$$
	h(x) \eqa \Phi(ax + d),\quad
	f(x) \eqa \Phi(ax + b),\quad
	g(x) \eqa \Phi(ax + c),
$$
for all $x\in\mathbb{R}$, or
$$
	h(x) \eqa 1_{ax + d\ge 0},\quad
	f(x) \eqa 1_{ax + b\ge 0},\quad
	g(x) \eqa 1_{ax + c\ge 0},
$$
for all $x\in\mathbb{R}$, for some $a,b,c,d\in\mathbb{R}$. It is now readily
verified by explicit computation as in the proof of sufficiency in 
Theorem \ref{thm:main} (see section \ref{sec:main}) that equality in
the Ehrhard-Borell inequality can only hold if $d=\lambda b+\mu c$,
completing the proof.
\end{proof}

\begin{rem}
The key idea behind the proof of Proposition \ref{prop:deg1d} was the 
dichotomy between \eqref{eq:N} and \eqref{eq:D}: we showed that the only 
possible equality cases in the nondegenerate case \eqref{eq:N} are 
\eqref{eq:H1} and \eqref{eq:H2}, while all three equality cases can 
appear in the degenerate case \eqref{eq:D}. It is interesting to note, 
however, that if \eqref{eq:H1} or \eqref{eq:H2} hold, the support of 
$(X_t,Y_t)$ must necessarily be degenerate under the assumption 
$\lambda+\mu=1$ of this section: by computing the explicit form of 
$b(t,x,y)$ and substituting into the differential equation of Lemma 
\ref{lem:fk}, it is readily verified that $X_t-Y_t$ is nonrandom and 
thus the support of $(X_t,Y_t)$ is contained in a (time-dependent) 
hyperplane. Thus it turns out \emph{a posteriori} that the nondegenerate 
situation \eqref{eq:N} can never occur when $\lambda+\mu=1$ and
there is equality in the Ehrhard-Borell inequality.
\end{rem}

\section{Induction on the dimension}
\label{sec:tech}

In the previous section, we settled the equality cases of the 
Ehrhard-Borell inequality on the real line $\mathbb{R}$ in the 
degenerate case $\lambda+\mu=1$. The restriction to the one-dimensional 
case considerably simplified the computations needed in the proof. The 
aim of the present section is to extend the main result of the previous 
section to any dimension $n$. That is, we will prove the following:

\begin{prop}
\label{prop:degnd}
Let $\lambda,\mu>0$ satisfy $\lambda+\mu=1$.
Let $f,g,h:\mathbb{R}^n\to[0,1]$ be nontrivial measurable
functions satisfying 
$$
	\Phi^{-1}(h(\lambda x+\mu y))
	\gea \lambda \Phi^{-1}(f(x)) + \mu \Phi^{-1}(g(y)).
$$
If equality holds in the Ehrhard-Borell inequality
$$
	\Phi^{-1}\bigg(\int h\, d\gamma_n\bigg) =
	\lambda\Phi^{-1}\bigg(\int f\, d\gamma_n\bigg) +
	\mu\Phi^{-1}\bigg(\int g\, d\gamma_n\bigg),
$$
then \textbf{either}
$$
	h(x) \eqa \Phi(\langle a,x\rangle + \lambda b+\mu c),\quad
	f(x) \eqa \Phi(\langle a,x\rangle + b),\quad
	g(x) \eqa \Phi(\langle a,x\rangle + c)
$$
for some $a\in\mathbb{R}^n$ and $b,c\in\mathbb{R}$, \textbf{or}
$$
	h(x) \eqa 1_{\langle a,x\rangle + \lambda b+\mu c\ge 0},\quad
	f(x) \eqa 1_{\langle a,x\rangle + b\ge 0},\quad
	g(x) \eqa 1_{\langle a,x\rangle + c\ge 0}
$$
for some $a\in\mathbb{R}^n$ and $b,c\in\mathbb{R}$, \textbf{or}
$$
	h(x) \eqa f(x) \eqa g(x) \eqa \Phi(V(x))
$$
for some concave function $V:\mathbb{R}^n\to\mathbb{\bar R}$.
\end{prop}

We will not give an independent proof of the $n$-dimensional case, but 
rather show that we can reduce this case to the one-dimensional case by 
induction on the dimension. In the induction step, it will be convenient 
to assume that the functions $f,g,h$ are smooth and take values in 
$(0,1)$. We therefore first prove the result under this regularity 
assumption in the following subsection, and then conclude the proof of 
Proposition \ref{prop:degnd} at the end of this section by a 
regularization argument.

\subsection{The regular case}

Our aim is to prove the following result.

\begin{prop}
\label{prop:degndsm}
Let $\lambda,\mu>0$ satisfy $\lambda+\mu=1$. Let
$f,g,h:\mathbb{R}^n\to(0,1)$ be smooth functions, and assume that
for every $x,y$
$$
	\Phi^{-1}(h(\lambda x+\mu y))
	\ge \lambda \Phi^{-1}(f(x)) + \mu \Phi^{-1}(g(y)).
$$
If equality holds in the Ehrhard-Borell inequality
$$
	\Phi^{-1}\bigg(\int h\, d\gamma_n\bigg) =
	\lambda\Phi^{-1}\bigg(\int f\, d\gamma_n\bigg) +
	\mu\Phi^{-1}\bigg(\int g\, d\gamma_n\bigg),
$$
then either
$$
	h(x) = \Phi(\langle a,x\rangle + \lambda b+\mu c),\quad
	f(x) = \Phi(\langle a,x\rangle + b),\quad
	g(x) = \Phi(\langle a,x\rangle + c)
$$
for all $x\in\mathbb{R}^n$ holds for some
$a\in\mathbb{R}^n$ and $b,c\in\mathbb{R}$, or
$$
	h(x) = f(x) = g(x) = \Phi(V(x))
$$
for all $x\in\mathbb{R}^n$ holds for
some concave function $V:\mathbb{R}^n\to\mathbb{R}$.
\end{prop}

The one-dimensional case $n=1$ of Proposition \ref{prop:degndsm} follows 
immediately from Proposition \ref{prop:deg1d}, using smoothness of 
$f,g,h$ to eliminate the null sets and the non-smooth equality case. We 
will presently prove the induction step: we assume in the rest of this
subsection that the statement of Proposition \ref{prop:degndsm} has been proved
in dimension $n-1$, and will show that it must also hold in dimension $n$.

To this end, fix smooth functions $f,g,h:\mathbb{R}^n\to(0,1)$ 
satisfying the condition of Proposition \ref{prop:degndsm}. Define for every
$z\in\mathbb{R}$ the functions $f_z,g_z,h_z:\mathbb{R}^{n-1}\to(0,1)$ as
$$
	h_z(x) := h(z,x),\qquad
	f_z(x) := f(z,x),\qquad
	g_z(x) := g(z,x),
$$
and define the functions $\hat f,\hat g,\hat h:\mathbb{R}\to(0,1)$ as
$$
	\hat h(z) := \int h_z\,d\gamma_{n-1},\qquad
	\hat f(z) := \int f_z\,d\gamma_{n-1},\qquad
	\hat g(z) := \int g_z\,d\gamma_{n-1}.
$$
Note that all the functions just defined are also smooth with values in
$(0,1)$.

\begin{lem}
\label{lem:indcases}
If equality holds in the Ehrhard-Borell inequality for $f,g,h$, 
either
\begin{equation}
\label{eq:I1}
\tag{I1}
	\hat h(z) = \Phi(az+\lambda b+\mu c),\qquad
	\hat f(z) = \Phi(az+b),\qquad
	\hat g(z) = \Phi(az+c)
\end{equation}
for some $a,b,c\in\mathbb{R}$, or
\begin{equation}
\label{eq:I2}
\tag{I2}
	\hat h(z)=\hat f(z)=\hat g(z) = \Phi(\hat V(z))
\end{equation}
for some concave function $\hat V:\mathbb{R}\to\mathbb{R}$.
\end{lem}

\begin{proof}
The assumption on $f,g,h$ implies that
$$
	\Phi^{-1}(h_{\lambda z_1+\mu z_2}(\lambda x+\mu y))
	\ge
	\lambda\Phi^{-1}(f_{z_1}(x))+
	\mu\Phi^{-1}(g_{z_2}(y)).
$$
Thus applying the Ehrhard-Borell inequality (Theorem \ref{thm:be}) yields
$$
	\Phi^{-1}(\hat h(\lambda z_1+\mu z_2)) \ge
	\lambda\Phi^{-1}(\hat f(z_1)) +
	\mu\Phi^{-1}(\hat g(z_2)),
$$
that is, $\hat f,\hat g,\hat h$ satisfy the assumption of the one-dimensional
Ehrhard-Borell inequality. On the other hand, as $\int f d\gamma_n =
\int \hat f d\gamma_1$ and analogously for $g,h$, the assumption of equality
in the Ehrhard-Borell inequality implies that
$$
	\Phi^{-1}\bigg(\int \hat h\,d\gamma_1\bigg) =
	\lambda \Phi^{-1}\bigg(\int \hat f\,d\gamma_1\bigg) +
	\mu \Phi^{-1}\bigg(\int \hat g\,d\gamma_1\bigg).
$$
The conclusion follows by applying Proposition \ref{prop:deg1d}.
\end{proof}

To proceed, we first address the first case of Lemma \ref{lem:indcases}.

\begin{lem}
\label{lem:i1}
If equality holds in the Ehrhard-Borell inequality and \eqref{eq:I1} holds, then
$$
	h(x) = \Phi(\langle a,x\rangle + \lambda b+\mu c),\quad
	f(x) = \Phi(\langle a,x\rangle + b),\quad
	g(x) = \Phi(\langle a,x\rangle + c)
$$
for some $a\in\mathbb{R}^n$ and $b,c\in\mathbb{R}$.
\end{lem}

\begin{proof}
By the definition of $\hat f(z)$, the assumption \eqref{eq:I1} implies
that
$$
	\Phi^{-1}\bigg(
	\int f_{z}\,d\gamma_{n-1}
	\bigg) =
	\Phi^{-1}(\hat f(z)) =
	 az +b,
$$
and analogously for $\hat g,\hat h$. Thus \eqref{eq:I1} implies
$$
	\Phi^{-1}\bigg(
	\int h_{\lambda z_1+\mu z_2}\,d\gamma_{n-1}
	\bigg) = 
	\lambda \Phi^{-1}\bigg(
	\int f_{z_1}\,d\gamma_{n-1}
	\bigg) +
	\mu\Phi^{-1}\bigg(
	\int g_{z_2}\,d\gamma_{n-1}
	\bigg).
$$
Recall that we showed in the proof of Lemma \ref{lem:indcases} that
the functions $f_{z_1},g_{z_2},h_{\lambda z_1+\mu z_2}$ satisfy the assumption
of the Ehrhard-Borell inequality on $\mathbb{R}^{n-1}$. As we assumed
at the outset of the proof of Proposition \ref{prop:degndsm} that its
conclusion holds in $n-1$ dimensions (the induction hypothesis), we conclude
that for every $z_1,z_2\in\mathbb{R}$, either
\begin{align*}
	&h_{\lambda z_1+\mu z_2}(x) =
	\Phi(\langle a_{z_1,z_2},x\rangle+\lambda b_{z_1,z_2}+\mu c_{z_1,z_2}),
	\\ &
	f_{z_1}(x) =
	\Phi(\langle a_{z_1,z_2},x\rangle+b_{z_1,z_2}),
	\\ &
	g_{z_2}(x) =
	\Phi(\langle a_{z_1,z_2},x\rangle+c_{z_1,z_2}).
\end{align*}
for some $a_{z_1,z_2}\in\mathbb{R}^{n-1}$ and $b_{z_1,z_2},c_{z_1,z_2}\in\mathbb{R}$,
or
$$
	h_{\lambda z_1+\mu z_2}(x)=f_{z_1}(x)=g_{z_2}(x)=
	\Phi(V_{z_1,z_2}(x))
$$
for some concave function $V_{z_1,z_2}:\mathbb{R}^{n-1}\to\mathbb{R}$.

Let us first argue that the second case can be ignored, so we may assume that
the first case holds for every $z_1,z_2\in\mathbb{R}$. To this end, suppose
the second case holds for some $z_1,z_2$. Integrating with respect to $x$
shows that we must then have $\hat f(z_1)=\hat g(z_2)$. Thus the second
case can only occur on the lower-dimensional set
$$
	D:=\{(z_1,z_2):az_1+b=az_2+c\}\subset\mathbb{R}^2.
$$
Consequently, any such $(z_1,z_2)$ can be approximated by 
$(z_1',z_2')\not\in D$ for which the first case must hold.
Now let $z_1'\to z_1$, $z_2'\to z_2$. By continuity of $f,g,h$,
it follows that $(\Phi^{-1}(h_{\lambda z_1+\mu 
z_2}),\Phi^{-1}(f_{z_1}), \Phi^{-1}(g_{z_2}))$ is a limit of triples
of linear functions with the same slope. But neither the slope nor the
offsets of these linear functions may diverge, as that would contradict the
assumption that $f,g,h$ take values in $(0,1)$. Thus
$\Phi^{-1}(h_{\lambda z_1+\mu z_2}),\Phi^{-1}(f_{z_1}),\Phi^{-1}(g_{z_2})$
are themselves linear functions with the same slope, that is, 
the first case is applies automatically even when $(z_1,z_2)\in D$.
We can therefore ignore the second case from now onward.

To complete the proof, it remains to understand the dependence of the
parameters $a_{z_1,z_2},b_{z_1,z_2},c_{z_1,z_2}$ on $z_1,z_2$. Let us
begin with $a_{z_1,z_2}$. Note that
$$
	\Phi^{-1}(f_{z_1}(x)) - \Phi^{-1}(f_{z_1}(-x)) =
	2\langle a_{z_1,z_2},x\rangle
$$
for all $x,z_1,z_2$. This shows that $a_{z_1,z_2}=a_{z_1}$ cannot
depend on $z_2$. Similarly,
$$
	\Phi^{-1}(g_{z_2}(x)) - \Phi^{-1}(g_{z_2}(-x)) =
	2\langle a_{z_1},x\rangle
$$
for all $x,z_1,z_2$, so that $a_{z_1}=a_0$ is simply a constant
independent of $z_1,z_2$.

By an entirely analogous argument, note that
$$
	\Phi^{-1}(h_{\lambda z_1+\mu z_2}(0)) =
	\lambda b_{z_1,z_2}+\mu c_{z_1,z_2},\quad
	\Phi^{-1}(f_{z_1}(0)) = b_{z_1,z_2},\quad
	\Phi^{-1}(g_{z_2}(0)) = c_{z_1,z_2}
$$
for all $x,z_1,z_2$. We conclude that $b_{z_1,z_2}=b_{z_1}$,
$c_{z_1,z_2}=c_{z_2}$, and
$$
	\lambda b_{z_1}+\mu c_{z_2} = w(\lambda z_1+\mu z_2)
$$
for some function $w:\mathbb{R}\to\mathbb{R}$. But as $f,g,h$ were assumed
to be smooth, the functions $z\mapsto b_z$, $z\mapsto c_z$, and
$z\mapsto w(z)$ must evidently be smooth as well. We can therefore
conclude using Lemma \ref{lem:eqcalc} that we have
$$
	b_z = a'z+b,\qquad c_z=a'z+c
$$
for some $a',b,c\in\mathbb{R}$.
Putting together the above observations, we conclude that
\begin{align*}
	h_z(x) &= \Phi(\langle a_0,x\rangle + a'z + \lambda b+\mu c),\\
	f_z(x) &= \Phi(\langle a_0,x\rangle + a'z + b),\\
	g_z(x) &= \Phi(\langle a_0,x\rangle + a'z + c).
\end{align*}
By the definition of $f_z,g_z,h_z$, this concludes the proof.
\end{proof}

It remains to consider the second case of Lemma \ref{lem:indcases}.

\begin{lem}
If equality holds in the Ehrhard-Borell inequality and \eqref{eq:I2} holds, then
$$
	h(x) = f(x) = g(x) = \Phi(V(x))
$$
for some concave function $V:\mathbb{R}^n\to\mathbb{R}$.
\end{lem}

\begin{proof}
The assumption \eqref{eq:I2} implies that
$$
	\Phi^{-1}(\hat h(z)) = \lambda \Phi^{-1}(\hat f(z)) +
	\mu \Phi^{-1}(\hat g(z)),
$$
which implies as in the proof of Lemma \ref{lem:i1} that
$$
	\Phi^{-1}\bigg(
	\int h_{z}\,d\gamma_{n-1}
	\bigg) = 
	\lambda \Phi^{-1}\bigg(
	\int f_{z}\,d\gamma_{n-1}
	\bigg) +
	\mu\Phi^{-1}\bigg(
	\int g_{z}\,d\gamma_{n-1}
	\bigg).
$$
Moreover, the functions $f_z,g_z,h_z$ clearly satisfy the assumption of
the Ehrhard-Borell inequality on $\mathbb{R}^{n-1}$. As our induction
hypothesis states that the conclusion of Proposition \ref{prop:degndsm} holds
in dimension $n-1$, we conclude that for every $z\in\mathbb{R}$ either
\begin{align*}
	&h_{z}(x) =
	\Phi(\langle a_{z},x\rangle+\lambda b_{z}+\mu c_{z}),
	\\ &
	f_{z}(x) =
	\Phi(\langle a_{z},x\rangle+b_{z}),
	\\ &
	g_{z}(x) =
	\Phi(\langle a_{z},x\rangle+c_{z}).
\end{align*}
for some $a_{z}\in\mathbb{R}^{n-1}$ and $b_{z},c_{z}\in\mathbb{R}$,
or
$$
	h_{z}(x)=f_{z}(x)=g_{z}(x)=
	\Phi(V_{z}(x))
$$
for some concave function $V_{z}:\mathbb{R}^{n-1}\to\mathbb{R}$.

Consider first $z\in\mathbb{R}$ for which the first case holds. 
Integrating with respect to $x$ and using that $\hat f(z)=\hat g(z)$ by 
\eqref{eq:I2}, we conclude that necessarily $b_z=c_z$. Thus the first 
case reduces to a special case of the second case, so there is no need 
to consider it separately. We will therefore ignore the first case from 
now onward.

As the second case holds for all $x\in\mathbb{R}^{n-1}$ and
$z\in\mathbb{R}$, we have shown that
$$
		h(x)=f(x)=g(x)
$$
everywhere. It remains to show that $\Phi^{-1}(f(x))$ is a concave function.
But note that the assumption of Proposition \ref{prop:degndsm} implies that
$$
	\Phi^{-1}(f(\lambda x+\mu y))
	\ge \lambda \Phi^{-1}(f(x)) + \mu \Phi^{-1}(f(y))
$$
for all $x,y$, so the claim follows as $\Phi^{-1}(f)$ is smooth.
\end{proof}

\subsection{Regularization}

To complete the proof of Proposition \ref{prop:degnd}, it remains to
show that the conclusion of Proposition \ref{prop:degndsm} continues to
hold in the absence of the additional regularity assumption. This is
easily accomplished by exploiting the techniques that we already used
to address regularity in the previous sections.

\begin{proof}[Proof of Proposition \ref{prop:degnd}]
Let $f,g,h$ satisfy the assumptions of Proposition \ref{prop:degnd}, and
assume equality holds in the Ehrhard-Borell inequality. First, we recall
from the proof of Corollary \ref{cor:ndeq} that $C(t,x,y)\ge 0$ for
every $t\in(0,1)$ and $x,y\in\mathbb{R}^n$. Thus
$$
	\Phi^{-1}(Q_th(\lambda x+\mu y)) \ge
	\lambda\Phi^{-1}(Q_tf(x))+\mu\Phi^{-1}(Q_tg(y)).
$$
Moreover, $Q_tf,Q_tg,Q_th$ are smooth and take values in $(0,1)$.
Thus we have shown that $Q_tf,Q_tg,Q_th$ satisfy the
assumptions of Proposition \ref{prop:degndsm} for every $t\in(0,1)$.

Let us now note that
$$
	\int Q_tf(x\sqrt{1-t})\,\gamma_n(dx) = 
	\int f\,d\gamma_n,
$$
and analogously for $g,h$. Thus if equality holds in the
Ehrhard-Borell inequality for $f,g,h$, then the same is true for
$Q_tf,Q_tg,Q_th$ (up to scaling the spatial coordinate).
We can therefore conclude from Proposition \ref{prop:degndsm} that
for every $t\in (0,1)$, either
$$
	Q_th(x) = \Phi(\langle a_t,x\rangle + \lambda b_t+\mu c_t),\quad
	Q_tf(x) = \Phi(\langle a_t,x\rangle + b_t),\quad
	Q_tg(x) = \Phi(\langle a_t,x\rangle + c_t)
$$
for some
$a_t\in\mathbb{R}^n$ and $b_t,c_t\in\mathbb{R}$, or
$$
	Q_th(x) = Q_tf(x) = Q_tg(x) = \Phi(V_t(x))
$$
for some concave function $V_t:\mathbb{R}^n\to\mathbb{R}$. If the first 
case holds for some $t\in(0,1)$, the conclusion of Proposition 
\ref{prop:degnd} follows immediately from Lemma \ref{lem:invert}. 
Conversely, if the second case holds for all $t\in(0,1)$, the proof is 
concluded by repeating the argument at the end of the proof of Lemma 
\ref{lem:concave}.
\end{proof}

\section{The general case}
\label{sec:gen}

So far, we have devoted all our efforts to proving some apparently 
rather special cases of the main result of this paper: Propositions 
\ref{prop:nondeg} and \ref{prop:degnd} prove Theorem \ref{thm:main} in 
the special case $m=2$, with $|\lambda_1-\lambda_2|\ne 1$, and without 
additional convexity assumptions. The aim of this final section of the 
paper is to complete the picture of the equality cases in the 
Ehrhard-Borell inequality. We will prove the equality cases in the 
remaining degenerate case $|\lambda_1-\lambda_2|=1$, extend the result 
to arbitrary $m$, and consider the additional cases that arise under 
convexity assumptions.

It turns out that none of these extensions lie at the core of the 
analysis of the equality cases. Rather, the equality cases we 
have proved in the previous sections will suffice to deduce the 
remaining cases of Theorem \ref{thm:main}. As we will see below, the 
degenerate case $|\lambda_1-\lambda_2|=1$ can be transformed to the 
degenerate case $\lambda_1+\lambda_2$ by a change of variables, so that 
these two situations are essentially in duality with one another. On the 
other hand, the extension to general $m$ and the treatment of the 
additional convexity assumptions can be deduced from the limited cases 
proved so far by an induction argument, which is similar in spirit to 
the treatment of the general Ehrhard-Borell inequality given in 
\cite{Bor08}. These arguments will be worked out in detail in the following
subsections, completing the proof of Theorem \ref{thm:main}.

\subsection{The degenerate case $|\lambda-\mu|=1$}

Even in the case of $m=2$ functions, we have so far neglected the 
remaining degenerate case $|\lambda-\mu|=1$ of Theorem \ref{thm:main}. 
This case will now be settled by the following lemma.

\begin{lem}
\label{lem:deglast}
Let $\lambda\ge\mu>0$ satisfy $\lambda=1+\mu$,
and let $f,g,h:\mathbb{R}^n\to[0,1]$ be nontrivial measurable
functions satisfying 
$$
	\Phi^{-1}(h(\lambda x+\mu y))
	\gea \lambda \Phi^{-1}(f(x)) + \mu \Phi^{-1}(g(y)).
$$
If equality holds in the Ehrhard-Borell inequality
$$
	\Phi^{-1}\bigg(\int h\, d\gamma_n\bigg) =
	\lambda\Phi^{-1}\bigg(\int f\, d\gamma_n\bigg) +
	\mu\Phi^{-1}\bigg(\int g\, d\gamma_n\bigg),
$$
then \textbf{either}
$$
	h(x) \eqa \Phi(\langle a,x\rangle + \lambda b+\mu c),\quad
	f(x) \eqa \Phi(\langle a,x\rangle + b),\quad
	g(x) \eqa \Phi(\langle a,x\rangle + c)
$$
for some $a\in\mathbb{R}^n$ and $b,c\in\mathbb{R}$, \textbf{or}
$$
	h(x) \eqa 1_{\langle a,x\rangle + \lambda b+\mu c\ge 0},\quad
	f(x) \eqa 1_{\langle a,x\rangle + b\ge 0},\quad
	g(x) \eqa 1_{\langle a,x\rangle + c\ge 0}
$$
for some $a\in\mathbb{R}^n$ and $b,c\in\mathbb{R}$, \textbf{or}
$$
	1-h(-x) \eqa 1-f(-x) \eqa g(x) \eqa \Phi(V(x))
$$
for some concave function $V:\mathbb{R}^n\to\mathbb{\bar R}$.
\end{lem}

\begin{proof}
Let us rearrange the assumption as
$$
	-\Phi^{-1}(f(x))
	\gea 
	\frac{\mu}{\lambda} \Phi^{-1}(g(y))
	-\frac{1}{\lambda}\Phi^{-1}(h(\lambda x+\mu y))
$$
(one may verify using the convention $\infty-\infty=-\infty$ that this 
claim is valid even when some of the terms take the values $\pm\infty$). 
Define
$$
	\tilde\lambda := \frac{\mu}{\lambda},\qquad
	\tilde\mu := \frac{1}{\lambda},\qquad
	\tilde x := -y,\qquad
	\tilde y := \lambda x+\mu y,
$$
and
$$
	\tilde h(x):=1-f(x),\qquad
	\tilde f(x):=g(-x),\qquad
	\tilde g(x):=1-h(x).
$$
Then $\tilde\lambda+\tilde\mu =1$ and
$$
	\Phi^{-1}(\tilde h(\tilde\lambda\tilde x + \tilde\mu\tilde y))
	\gea 
	\tilde\lambda \Phi^{-1}(\tilde f(\tilde x))
	+\tilde\mu \Phi^{-1}(\tilde g(\tilde y)),
$$
where we used $-\Phi^{-1}(x)=\Phi^{-1}(1-x)$. Moreover, equality in
the Ehrhard-Borell inequality implies, after rearranging, that
$$
	\Phi^{-1}\bigg(\int \tilde h\, d\gamma_n\bigg) =
	\tilde\lambda\Phi^{-1}\bigg(\int \tilde f\, d\gamma_n\bigg) +
	\tilde\mu\Phi^{-1}\bigg(\int \tilde g\, d\gamma_n\bigg),
$$
where we used that the Gaussian measure $\gamma_n$ is symmetric.
Thus we have reduced to the dual degenerate case $\tilde\lambda+\tilde\mu=1$,
for which Proposition \ref{prop:degnd} implies that either
$$
	\tilde h(x) \eqa \Phi(\langle a,x\rangle + \tilde \lambda b+\tilde \mu c),\quad
	\tilde f(x) \eqa \Phi(\langle a,x\rangle + b),\quad
	\tilde g(x) \eqa \Phi(\langle a,x\rangle + c)
$$
for some $a\in\mathbb{R}^n$ and $b,c\in\mathbb{R}$, or
$$
	\tilde h(x) \eqa 1_{\langle a,x\rangle + \tilde\lambda b+\tilde\mu c\ge 0},\quad
	\tilde f(x) \eqa 1_{\langle a,x\rangle + b\ge 0},\quad
	\tilde g(x) \eqa 1_{\langle a,x\rangle + c\ge 0}
$$
for some $a\in\mathbb{R}^n$ and $b,c\in\mathbb{R}$, or
$$
	\tilde h(x) \eqa \tilde f(x) \eqa \tilde g(x) \eqa \Phi(V(x))
$$
for some concave function $V:\mathbb{R}^n\to\mathbb{\bar R}$. Substituting
the definitions of $\tilde f,\tilde g,\tilde h,\tilde\lambda,\tilde\mu$,
and using $1-\Phi(x)=\Phi(-x)$ and $1_{\langle a,x\rangle+b>0}\eqa
1_{\langle a,x\rangle+b\ge 0}$, concludes the proof.
\end{proof}

\subsection{Extension to general $m\ge 3$}

In this subsection we will prove Theorem \ref{thm:main} in full, with 
the exception of the additional cases where 
$\Phi^{-1}(h),\Phi^{-1}(f_i)$ are assumed to be a.e.\ concave which will 
be treated in the next subsection.

By virtue of Propositions \ref{prop:nondeg} and \ref{prop:deg1d} and of
Lemma \ref{lem:deglast}, the result of this section has been proved for $m=2$.
We therefore proceed by induction. In the remainder of this subsection,
we will assume that the conclusion has already been proved for $m-1$
functions, and show that the conclusion must then extend to $m$ functions.

In the following, we fix $\lambda_1\ge \lambda_2\ge\cdots\ge\lambda_m>0$
satisfying \eqref{eq:A}, as well as nontrivial measurable functions
$h,f_1,\ldots,f_m:\mathbb{R}^n\to[0,1]$ satisfying \eqref{eq:B}. We will
also assume throughout that equality holds in the Ehrhard-Borell inequality
$$
	\Phi^{-1}\bigg(\int h\,d\gamma_n\bigg)=
	\sum_{i\le m}\lambda_i\Phi^{-1}\bigg(
	\int f_i\,d\gamma_n
	\bigg),
$$
and proceed to deduce the resulting equality cases listed in Theorem
\ref{thm:main}.

Our ability to apply the induction hypothesis relies on the following
observation.

\begin{lem}
\label{lem:feas}
One can choose $\lambda=\lambda_1$ and $\mu>0$ such that
$$
	\lambda+\mu\ge 1,\qquad\quad |\lambda-\mu|\le 1,
$$
and
$$
	\sum_{2\le i\le m}\frac{\lambda_i}{\mu}\ge 1,\qquad\quad
	\frac{\lambda_2}{\mu}-
	\sum_{3\le i\le m}\frac{\lambda_i}{\mu}<1.
$$
That is, the families of coefficients $(\lambda,\mu)$ and
$(\lambda_2/\mu,\ldots,\lambda_m/\mu)$ both satisfy \eqref{eq:A}.
\end{lem}

\begin{proof}
We claim that we can choose
$$
	\mu = \min\Bigg(1+\lambda_1,
	\sum_{2\le i\le m}\lambda_i\Bigg).
$$
Indeed, suppose first that $\mu = \sum_{i\ge 2}\lambda_i\le 1+\lambda_1$.
Then 
$$
	\sum_{2\le i\le m}\frac{\lambda_i}{\mu}= 1,\qquad\quad
	\frac{\lambda_2}{\mu}-
	\sum_{3\le i\le m}\frac{\lambda_i}{\mu}< 
	\frac{\lambda_2}{\mu}<1.
$$
But $\lambda+\mu\ge 1$ and $\lambda-\mu\le 1$ follow from the assumption that $(\lambda_1,\ldots,\lambda_m)$ satisfy \eqref{eq:A},
while $\mu-\lambda\le 1$ follows the assumption that $\mu\le 1+\lambda_1$.

Now suppose $\mu=1+\lambda_1\le \sum_{i\ge 2}\lambda_i$.
Then $\lambda+\mu\ge 1$ and $|\lambda-\mu|=1$ follow trivially. On the
other hand, we have $\sum_{i\ge 2}\lambda_i/\mu\ge 1$
by assumption. Finally,
$$
	\frac{\lambda_2}{\mu}-
	\sum_{3\le i\le m}\frac{\lambda_i}{\mu}
	=
	\frac{2\lambda_2-\lambda_1}{\mu}
	+ \frac{\lambda_1}{\mu}
	- \sum_{2\le i\le m}\frac{\lambda_i}{\mu}
	\le 
	\frac{1-\lambda_1+2\lambda_2}{\mu}
	<1
$$
as $(\lambda_1,\ldots,\lambda_m)$ satisfy \eqref{eq:A},
provided that $\lambda_1>\lambda_2$. But if $\lambda_1=\lambda_2$,
$$
	\frac{\lambda_2}{\mu}-
	\sum_{3\le i\le m}\frac{\lambda_i}{\mu}
	=
	\frac{\lambda_1}{\mu}-
	\sum_{3\le i\le m}\frac{\lambda_i}{\mu}
	< \frac{\lambda_1}{\mu}<1,
$$
and the proof is complete.
\end{proof}

We are now ready to proceed to the main argument.

\begin{proof}[Proof of Theorem \ref{thm:main} (without convexity assumptions)]
We adopt the notation and assumptions stated at the beginning of this
subsection. Choose $\lambda,\mu$ as in Lemma \ref{lem:feas} and
define $\tilde\lambda_i:=\lambda_i/\mu$ for $i\ge 2$.
Define the function
$$
	\Phi^{-1}(\tilde h(x)) :=
	\mathop{\mathrm{ess\,sup}}_{z_3,\ldots,z_m\in\mathbb{R}^n}
	\bigg\{
	\tilde\lambda_2\Phi^{-1}\bigg(f_2\bigg(
	\frac{x-\sum_{i\ge 3}\tilde\lambda_iz_i}{\tilde\lambda_2}
	\bigg)\bigg)
	+
	\sum_{i=3}^m \tilde\lambda_i
	\Phi^{-1}(f_i(z_i))
	\bigg\}.
$$
This definition is made so that, on the one hand,
$$
	\Phi^{-1}(h(\lambda x+\mu y)) \gea
	\lambda \Phi^{-1}(f_1(x)) +
	\mu \Phi^{-1}(\tilde h(y))
$$
by assumption \eqref{eq:B}, while on the other hand, by definition
$$
	\Phi^{-1}\Bigg(
	\tilde h\Bigg(
	\sum_{2\le i\le m}\tilde\lambda_i z_i\Bigg)\Bigg)
	\gea
	\sum_{2\le i\le m}\tilde\lambda_i
	\Phi^{-1}(f_i(z_i)).
$$
Using the Ehrhard-Borell inequality twice, we obtain
\begin{align*}
	\Phi^{-1}\bigg(\int h\,d\gamma_n\bigg)
	&\ge
	\lambda\Phi^{-1}\bigg(\int f_1\,d\gamma_n\bigg)
	+\mu\Phi^{-1}\bigg(\int \tilde h\,d\gamma_n\bigg)
	\\ &\ge
	\sum_{1\le i\le m}\lambda_i\Phi^{-1}\bigg(\int f_i\,d\gamma_n\bigg).
\end{align*}
Here the first inequality is the Ehrhard-Borell inequality for $2$ functions,
while the second is the Ehrhard-Borell inequality for $m-1$ functions.
But as we assumed equality in the Ehrhard-Borell inequality for
$h,f_1,\ldots,f_m$, it follows that both these inequalities must be
equality. Thus we can use the case of Theorem \ref{thm:main} that
we already proved, together with the induction hypothesis, to conclude
the following:
\begin{enumerate}[$\bullet$]
\itemsep\abovedisplayskip
\item Either
\begin{equation}
\label{eq:M1}
\tag{M1}
	h(x) \eqa \Phi(\langle a,x\rangle + \lambda b+\mu c),\quad
	f_1(x) \eqa \Phi(\langle a,x\rangle + b),\quad
	\tilde h(x) \eqa \Phi(\langle a,x\rangle + c),
\end{equation}
or
\begin{equation}
\label{eq:M2}
\tag{M2}
	h(x) \eqa 1_{\langle a,x\rangle + \lambda b+\mu c\ge 0},\quad
	f_1(x) \eqa 1_{\langle a,x\rangle + b\ge 0},\quad
	\tilde h(x) \eqa 1_{\langle a,x\rangle + c\ge 0};
\end{equation}
or, if $\lambda+\mu=1$,
\begin{equation}
\label{eq:M3}
\tag{M3}
	h(x) \eqa f_1(x) \eqa \tilde h(x) \eqa \Phi(V(x))
\end{equation}
for some concave function $V:\mathbb{R}^n\to\mathbb{\bar R}$;
or, if $\lambda=1+\mu$,
\begin{equation}
\label{eq:M4}
\tag{M4}
	1-h(-x) \eqa 1-f_1(-x) \eqa \tilde h(x) \eqa \Phi(V(x))
\end{equation}
for some concave function $V:\mathbb{R}^n\to\mathbb{\bar R}$;
or, if $\mu=1+\lambda$,
\begin{equation}
\label{eq:M5}
\tag{M5}
	1-h(-x) \eqa 1-\tilde h(-x) \eqa f_1(x) \eqa \Phi(V(x))	
\end{equation}
for some concave function $V:\mathbb{R}^n\to\mathbb{\bar R}$.
\item Either
\begin{equation}
\label{eq:M1'}
\tag{M1'}
	\tilde h(x) \eqa \Phi(\langle a,x\rangle + b),\qquad
	f_i(x) \eqa \Phi(\langle a,x\rangle + b_i)
\end{equation}
for all $i\ge 2$ or
\begin{equation}
\label{eq:M2'}
\tag{M2'}
	\tilde h(x) \eqa 1_{\langle a,x\rangle + b\ge 0},\qquad
	f_i(x) \eqa 1_{\langle a,x\rangle + b_i\ge 0}
\end{equation}
for all $i\ge 2$, where $b=\sum_{i\ge 2}\tilde\lambda_ib_i$; 
or, if $\sum_{i\ge 2}\tilde\lambda_i=1$,
\begin{equation}
\label{eq:M3'}
\tag{M3'}
	\tilde h(x)\eqa f_2(x)\eqa\ldots\eqa f_m(x)
\end{equation}
and $\Phi^{-1}(\tilde h),\Phi^{-1}(f_2),\ldots,\Phi^{-1}(f_m)$ are 
a.e.\ concave.
\end{enumerate}
\vskip.1cm
Completing the proof is now a matter of considering every possible
combination of these different cases. Let us consider each one in turn.
\begin{enumerate}[$\bullet$]
\itemsep\abovedisplayskip
\item \textbf{Case \eqref{eq:M1'}:} in this case \eqref{eq:M2} cannot
occur, as it contradicts the given form of $\tilde h$. 
However, each of 
the remaining cases \eqref{eq:M1}, \eqref{eq:M3}, \eqref{eq:M4}, \eqref{eq:M5}
will give rise to the equality case \eqref{eq:H1} of Theorem \ref{thm:main}, as
is readily verified by substituting the given form of $\tilde h$ and
using the identity $1-\Phi(x)=\Phi(-x)$.
\item \textbf{Case \eqref{eq:M2'}:} in this case \eqref{eq:M1} cannot
occur, as it contradicts the given form of $\tilde h$.
However, each of 
the remaining cases \eqref{eq:M2}, \eqref{eq:M3}, \eqref{eq:M4}, \eqref{eq:M5}
will give rise to the equality case \eqref{eq:H2} of Theorem \ref{thm:main}
by the same reasoning as above.
\item \textbf{Case \eqref{eq:M3'}:} 
If \eqref{eq:M1} or \eqref{eq:M2} holds, we readily obtain the equality cases
\eqref{eq:H1} or \eqref{eq:H2} in Theorem \ref{thm:main}, respectively.
If \eqref{eq:M3} holds, we obtain the equality case
$$
	h(x)\eqa f_1(x)\eqa\ldots\eqa f_m(x)
$$
and $\Phi^{-1}(h),\Phi^{-1}(f_1),\ldots,\Phi^{-1}(f_m)$ are a.e.\ concave.
However, note that this can only occur when $\lambda+\mu=1$ and
$\sum_{i\ge 2}\tilde\lambda_i=1$, which implies that
$\sum_{i\ge 1}\lambda_i=1$.
If \eqref{eq:M4} holds, we readily obtain the equality case
$$
	1-h(-x)\eqa 1-f_1(-x)\eqa f_2(x)\eqa\ldots\eqa f_m(x)
$$
and $\Phi^{-1}(f_2),\ldots,\Phi^{-1}(f_m)$ are a.e.\ concave.
However, note that this can only occur when $\lambda=1+\mu$ and
$\sum_{i\ge 2}\tilde\lambda_i=1$, which implies that
$\lambda_1-\sum_{i\ge 2}\lambda_i=1$.
Finally, if \eqref{eq:M5} holds, then 
$\Phi^{-1}(\tilde h(x))$ and
$-\Phi^{-1}(\tilde h(-x))$ must both be a.e.\ concave, which
implies that either \eqref{eq:H1} or \eqref{eq:H2} must hold by 
Lemma \ref{lem:conc} below.
\end{enumerate}
\vskip.1cm
As we have considered all possible cases, the proof is complete.
\end{proof}

It remains to establish the following fact that was used above.

\begin{lem}
\label{lem:conc}
If $V:\mathbb{R}^n\to\mathbb{\bar R}$ is a measurable function
so that $V(x)$ and $-V(x)$ are a.e.\ concave, then
$V(x)\eqa\langle a,x\rangle + b$ or $V(x)\eqa\Phi^{-1}(1_{\langle a,x\rangle+b\ge 0})$
for some $a,b$.
\end{lem}

\begin{proof}
Suppose first that $|V(x)|<\infty$ occurs on a set with positive measure.
Then by \cite[Theorem 3]{Dub77} we may assume that $V$ is a
proper concave function. In particular, $V$ is continuous on its domain.
Now note that by continuity, a.e.\ concavity of $-V$ implies that
$-V$ is also concave. Thus $V(x)=\langle a,x\rangle+b$ must be affine.

Now suppose $|V(x)|\eqa\infty$. Then $V(x)\eqa\Phi^{-1}(1_K(x))$ for a 
convex set $K$. The assumption that $-V(x)\eqa\Phi^{-1}(1_{K^c}(x))$ is 
a.e.\ concave now implies that $K^c$ must differ from some convex set 
$\tilde K$ by a null set. In particular, $K\cap\tilde K$ is a null set, 
so $K,\tilde K$ can only intersect on their boundaries. By the 
Hahn-Banach theorem, there exist $a,b$ such that $K\subseteq\{x:\langle 
a,x\rangle+b\ge 0\}$ and $\tilde K\subseteq\{x:\langle a,x\rangle+b\le 
0\}$. But as $(K\cup\tilde K)^c$ is a null set, we have
$1_K(x)\eqa 1_{\langle a,x\rangle+b\ge 0}$ and the proof is complete.
\end{proof}

\subsection{The convex case}

We finally address the last part of Theorem \ref{thm:main}, which is 
concerned with the equality cases in the Ehrhard-Borell inequality under 
additional convexity assumptions. In this case, the Ehrhard-Borell 
inequality is valid for any $\sum_i\lambda_i\ge 1$ (that is, the second 
assumption in \eqref{eq:A} is not needed). Note, however, that in the 
case $\sum_i\lambda_i=1$ the convex equality cases are already fully 
settled by the general part of Theorem \ref{thm:main}, so that it 
remains to consider the case $\sum_i\lambda_i>1$.

\begin{proof}[Proof of Theorem \ref{thm:main} under convexity 
assumptions] As indicated above, it suffices to assume that 
$\lambda:=\sum_i\lambda_i>1$. Let 
$h,f_1,\ldots,f_m:\mathbb{R}^n\to[0,1]$ be nontrivial measurable 
functions satisfying \eqref{eq:B}, and in addition that 
$\Phi^{-1}(h),\Phi^{-1}(f_1),\ldots,\Phi^{-1}(f_m)$ are a.e.\ concave.
We also assume equality holds in the Ehrhard-Borell inequality.

Define the function $\tilde h$ according to
$$
	\Phi^{-1}(\tilde h(x)) := \frac{\Phi^{-1}(h(\lambda x))}{\lambda}.
$$
Using that $\Phi^{-1}(h)$ is a.e.\ concave, we can write
$$
	\Phi^{-1}\bigg(h\bigg(\frac{\lambda}{2}x+\frac{\lambda}{2}y
	\bigg)\bigg)
	\gea
	\frac{\lambda}{2}\Phi^{-1}(\tilde h(x)) +
	\frac{\lambda}{2}\Phi^{-1}(\tilde h(y)).
$$
As $\lambda>1$, the Ehrhard-Borell inequality yields
$$
	\Phi^{-1}\bigg(\int h\,d\gamma_n\bigg) \ge
	\lambda\Phi^{-1}\bigg(\int \tilde h\,\gamma_n\bigg).
$$
On the other hand, the assumption \eqref{eq:B} implies that
$$
	\Phi^{-1}\Bigg(\tilde h\Bigg(
	\sum_{i\le m}\frac{\lambda_i}{\lambda}x_i\Bigg)\Bigg) 
	\gea
	\sum_{i\le m} \frac{\lambda_i}{\lambda}
	\Phi^{-1}(f_i(x_i)).
$$
As $\sum_i\lambda_i/\lambda=1$,
we can use again the Ehrhard-Borell inequality to obtain
$$
	\Phi^{-1}\bigg(\int h\,d\gamma_n\bigg) \ge
	\lambda\Phi^{-1}\bigg(\int \tilde h\,\gamma_n\bigg)
	\ge
	\sum_{i\le m}\lambda_i\Phi^{-1}\bigg(
	\int f_i\,d\gamma_n
	\bigg).
$$
This proves the Ehrhard-Borell inequality in the convex case. However,
as we assumed equality holds in the Ehrhard-Borell inequality for
$h,f_1,\ldots,f_m$, both intermediate applications of the
Ehrhard-Borell inequality must yield equality as well. Let us
consider each of these equality cases.
The first inequality applies only the nondegenerate case of the
Ehrhard-Borell inequality as $\lambda/2+\lambda/2>1$ and
$|\lambda/2-\lambda/2|<1$. Therefore, Proposition \ref{prop:nondeg}
shows that either
$$
	h(x) \eqa \Phi(\langle a,x\rangle+\lambda b),\qquad
	\tilde h(x) \eqa \Phi(\langle a,x\rangle+b),
$$
or
$$
	h(x) \eqa 1_{\langle a,x\rangle+\lambda b\ge 0},\qquad
	\tilde h(x) \eqa 1_{\langle a,x\rangle+b\ge 0},
$$
for some $a\in\mathbb{R}^n$, $b\in\mathbb{R}$. On the other hand,
the equality cases resulting from the second inequality as given
by the general case of Theorem \ref{thm:main} as follows: either
$$
	\tilde h(x) \eqa \Phi(\langle a,x\rangle + b),\qquad
	f_i(x) \eqa \Phi(\langle a,x\rangle + b_i)
$$
for all $i$, or
$$
	\tilde h(x) \eqa 1_{\langle a,x\rangle + b\ge 0},\qquad
	f_i(x) \eqa 1_{\langle a,x\rangle + b_i\ge 0}
$$
for all $i$, for some $a\in\mathbb{R}^n$, $b_1,\ldots,b_m\in\mathbb{R}$,
and $b=\sum_i\lambda_ib_i/\lambda$; or
$$
	\tilde h(x)\eqa f_1(x)\eqa\ldots\eqa f_m(x).
$$
All these cases are readily verified to result in the equality cases
\eqref{eq:H1} or \eqref{eq:H2}. We have therefore completed the proof
of Theorem \ref{thm:main}.
\end{proof}

\subsection*{Acknowledgments}

This work was supported in part by NSF grant CAREER-DMS-1148711 and by the 
ARO through PECASE award W911NF-14-1-0094. The paper was completed while 
the authors were in residence at the Mathematical Sciences Research 
Institute in Berkeley, California, supported by NSF grant DMS-1440140. The 
hospitality of MSRI and of the organizers of the program on Geometric 
Functional Analysis and Applications is gratefully acknowledged. Finally, 
we thank a referee for helpful comments that improved the presentation of 
this paper.


\begin{thebibliography}{10}

\bibitem{And15}
B.~Andrews.
\newblock Moduli of continuity, isoperimetric profiles, and multi-point
  estimates in geometric heat equations.
\newblock In {\em Surveys in differential geometry 2014. {R}egularity and
  evolution of nonlinear equations}, volume~19 of {\em Surv. Differ. Geom.},
  pages 1--47. Int. Press, Somerville, MA, 2015.

\bibitem{BGL14}
D.~Bakry, I.~Gentil, and M.~Ledoux.
\newblock {\em Analysis and Geometry of {M}arkov Diffusion Operators}.
\newblock Springer, 2014.

\bibitem{Ban98}
W.~Banaszczyk.
\newblock Balancing vectors and {G}aussian measures of {$n$}-dimensional convex
  bodies.
\newblock {\em Random Structures Algorithms}, 12(4):351--360, 1998.

\bibitem{BS97}
W.~Banaszczyk and S.~J. Szarek.
\newblock Lattice coverings and {G}aussian measures of {$n$}-dimensional convex
  bodies.
\newblock {\em Discrete Comput. Geom.}, 17(3):283--286, 1997.

\bibitem{Bar06}
F.~Barthe.
\newblock The {B}runn-{M}inkowski theorem and related geometric and functional
  inequalities.
\newblock In {\em Proceedings of the ICM}. European Mathematical Society, 2006.

\bibitem{BH09}
F.~Barthe and N.~Huet.
\newblock On {G}aussian {B}runn-{M}inkowski inequalities.
\newblock {\em Studia Math.}, 191(3):283--304, 2009.

\bibitem{Bor74}
C.~Borell.
\newblock Convex set functions in {$d$}-space.
\newblock {\em Period. Math. Hungar.}, 6(2):111--136, 1975.

\bibitem{Bor96}
C.~Borell.
\newblock A note on parabolic convexity and heat conduction.
\newblock {\em Ann. Inst. H. Poincar\'e Probab. Statist.}, 32(3):387--393,
  1996.

\bibitem{Bor03}
C.~Borell.
\newblock The {E}hrhard inequality.
\newblock {\em C. R. Math. Acad. Sci. Paris}, 337:663–666, 2003.

\bibitem{Bor08}
C.~Borell.
\newblock Inequalities of the {B}runn-{M}inkowski type for {G}aussian measures.
\newblock {\em Probab. Theory Related Fields}, 140:195--205, 2008.

\bibitem{BL76}
H.~J. Brascamp and E.~H. Lieb.
\newblock On extensions of the {B}runn-{M}inkowski and {P}r\'ekopa-{L}eindler
  theorems, including inequalities for log-concave functions, and with an
  application to the diffusion equation.
\newblock {\em J. Funct. Anal.}, 22:366--389, 1976.

\bibitem{Bur94}
A.~Burchard.
\newblock Cases of equality in the {R}iesz rearrangement inequality.
\newblock {\em Ann. Math.}, 143:499--527, 1994.

\bibitem{CK01}
E.~Carlen and C.~Kerce.
\newblock On the cases of equality in {B}obkov's inequality and {G}aussian
  rearrangement.
\newblock {\em Calc. Var.}, 13:1--18, 2001.

\bibitem{Col05}
A.~Colesanti.
\newblock {B}runn-{M}inkowski inequalities for variational functionals and
  related problems.
\newblock {\em Adv. Math.}, 194:105--140, 2005.

\bibitem{Dub77}
S.~Dubuc.
\newblock Crit\`eres de convexit\'e et in\'egalit\'es int\'egrales.
\newblock {\em Ann. Inst. Fourier}, 27:135--165, 1977.

\bibitem{Ehr83}
A.~Ehrhard.
\newblock Sym\'etrisation dans l'espace de {G}auss.
\newblock {\em Math. Scand.}, 53:281--301, 1983.

\bibitem{Ehr86}
A.~Ehrhard.
\newblock {\'E}l\'ements extr\'emaux pour les in\'egalit\'es de
  {B}runn-{M}inkowski {G}aussiennes.
\newblock {\em Ann. Inst. H. Poincar\'e Probab. Statist.}, 22(2):149--168,
  1986.

\bibitem{Ev97}
L.~C. Evans.
\newblock {\em Partial Differential Equations}, volume~19 of {\em Graduate
  Studies in Mathematics}.
\newblock AMS, 1997.

\bibitem{Fig14}
A.~Figalli.
\newblock Quantitative stability results for the {B}runn-{M}inkowski
  inequality.
\newblock In {\em Proceedings of the ICM}, 2014.

\bibitem{Fol99}
G.~B. Folland.
\newblock {\em Real analysis}.
\newblock Pure and Applied Mathematics (New York). John Wiley \& Sons, Inc.,
  New York, second edition, 1999.
\newblock Modern techniques and their applications, A Wiley-Interscience
  Publication.

\bibitem{Gar02}
R.~J. Gardner.
\newblock The {B}runn-{M}inkowski inequality.
\newblock {\em Bull. Amer. Math. Soc.}, 39:355--405, 2002.

\bibitem{GMRS17}
N.~Gozlan, M.~Madiman, C.~Roberto, and P.~M. Samson.
\newblock Deviation inequalities for convex functions motivated by the
  {T}alagrand conjecture.
\newblock {\em Zap. Nauchn. Sem. S.-Peterburg. Otdel. Mat. Inst. Steklov.
  (POMI)}, 457:168--182, 2017.

\bibitem{HM53}
R.~Henstock and A.~M. Mc{B}eath.
\newblock On the measure of sum sets, {I}. the theorems of {B}runn, {M}inkowski
  and {L}usternik.
\newblock {\em Proc. London Math. Soc.}, 3:182--194, 1953.

\bibitem{HLT17}
R.~H\"opfner, E.~L\"ocherbach, and M.~Thieullen.
\newblock Strongly degenerate time inhomogeneous {SDE}s: densities and support
  properties. {A}pplication to {H}odgkin-{H}uxley type systems.
\newblock {\em Bernoulli}, 23(4A):2587--2616, 2017.

\bibitem{Hor05}
P.~H\"orfelt.
\newblock The moment problem for some {W}iener functionals: corrections to
  previous proofs.
\newblock {\em J. Appl. Probab.}, 42(3):851--860, 2005.
\newblock With an appendix by H. L. Pedersen.

\bibitem{IW89}
N.~Ikeda and S.~Watanabe.
\newblock {\em Stochastic differential equations and diffusion processes},
  volume~24 of {\em North-Holland Mathematical Library}.
\newblock North-Holland Publishing Co., Amsterdam; Kodansha, Ltd., Tokyo,
  second edition, 1989.

\bibitem{Iv16}
P.~Ivanisvili.
\newblock Boundary value problem and the {E}hrhard inequality, 2016.
\newblock Preprint arxiv:1605.04840.

\bibitem{IV15}
P.~Ivanisvili and A.~Volberg.
\newblock Bellman partial differential equation and the hill property for
  classical isoperimetric problems, 2015.
\newblock Preprint arxiv:1506.03409.

\bibitem{KS88}
I.~Karatzas and S.~E. Shreve.
\newblock {\em Brownian motion and stochastic calculus}.
\newblock Springer, 1988.

\bibitem{KM17}
A.~V. Kolesnikov and E.~Milman.
\newblock Sharp {P}oincar\'e-type inequality for the {G}aussian measure on the
  boundary of convex sets.
\newblock In {\em Geometric aspects of functional analysis}, volume 2169 of
  {\em Lecture Notes in Math.} Springer, Berlin, 2017.
\newblock To appear.

\bibitem{Kun78}
H.~Kunita.
\newblock Supports of diffusion processes and controllability problems.
\newblock In {\em Proceedings of the {I}nternational {S}ymposium on
  {S}tochastic {D}ifferential {E}quations ({R}es. {I}nst. {M}ath. {S}ci.,
  {K}yoto {U}niv., {K}yoto, 1976)}, pages 163--185. Wiley, New
  York-Chichester-Brisbane, 1978.

\bibitem{Lat96}
R.~Lata{\l}a.
\newblock A note on the {E}hrhard inequality.
\newblock {\em Studia Math.}, 118:169--174, 1996.

\bibitem{Lat02}
R.~Lata{\l}a.
\newblock On some inequalities for {G}aussian measures.
\newblock In {\em Proceedings of the ICM}. Higher Ed. Press, 2002.

\bibitem{LO99}
R.~Lata{\l}a and K.~Oleszkiewicz.
\newblock Gaussian measures of dilatations of convex symmetric sets.
\newblock {\em Ann. Probab.}, 27(4):1922--1938, 1999.

\bibitem{Led96}
M.~Ledoux.
\newblock Isoperimetry and {G}aussian analysis.
\newblock In {\em Lectures on probability theory and statistics (Saint-Flour,
  1994)}, volume 1648 of {\em Lecture Notes in Math.}, pages 165--294.
  Springer, 1996.

\bibitem{LS01}
R.~S. Liptser and A.~N. Shiryaev.
\newblock {\em Statistics of random processes. {I}}, volume~5 of {\em
  Applications of Mathematics (New York)}.
\newblock Springer-Verlag, Berlin, expanded edition, 2001.

\bibitem{Lus35}
L.~A. Lusternik.
\newblock Die {B}runn-{M}inkowskische {U}ngleichung f\"ur beliebige messbare
  {M}engen.
\newblock {\em C. R. (Doklady) Acad. Sci. USSR}, 8:55--58, 1935.

\bibitem{Min10}
H.~Minkowski.
\newblock {\em Geometrie der {Z}ahlen}.
\newblock B. G. Teubner, Leipzig und Berlin, 1910.

\bibitem{NP16}
J.~Neeman and G.~Paouris.
\newblock An interpolation proof of {E}hrhard's inequality, 2016.
\newblock Preprint arxiv:1605.07233.

\bibitem{NvdS16}
H.~Nijmeijer and A.~{Van Der Schaft}.
\newblock {\em Nonlinear Dynamical Control Systems}.
\newblock Springer, 2016.
\newblock Corrected printing.

\bibitem{PV17}
G.~Paouris and P.~Valettas.
\newblock A {G}aussian small deviation inequality for convex functions.
\newblock {\em Ann. Probab.}, 2017.
\newblock To appear.

\bibitem{Par67}
K.~R. Parthasarathy.
\newblock {\em Probability measures on metric spaces}.
\newblock Probability and Mathematical Statistics, No. 3. Academic Press, Inc.,
  New York-London, 1967.

\bibitem{RW94}
L.~C.~G. Rogers and D.~Williams.
\newblock {\em Diffusions, {M}arkov Processes and Martingales. {II}: {I}t\^o
  calculus}.
\newblock Cambridge University Press, second edition, 1994.

\bibitem{Rud91}
W.~Rudin.
\newblock {\em Functional Analysis}.
\newblock McGraw-Hill, second edition, 1991.

\bibitem{Sch93}
R.~Schneider.
\newblock {\em Convex bodies: the {B}runn-{M}inkowski theory}.
\newblock Cambridge, 1993.

\bibitem{SV72}
D.~W. Stroock and S.~R.~S. Varadhan.
\newblock On the support of diffusion processes with applications to the strong
  maximum principle.
\newblock In {\em Proceedings of the Sixth Berkeley Symposium on Mathematical
  Statistics and Probability, Volume 3: Probability Theory}, pages 333--359,
  Berkeley, Calif., 1972. University of California Press.

\bibitem{Val17}
P.~Valettas.
\newblock On the tightness of {G}aussian concentration for convex functions.
\newblock {\em J. d'Analyse Math.}, 2017.
\newblock To appear.

\bibitem{vH17}
R.~{Van Handel}.
\newblock The {B}orell-{E}hrhard game.
\newblock {\em Probab. Theory Rel. Fields}, 170:555--585, 2018.

\end{thebibliography}

\end{document}